\documentclass[reqno,11pt]{amsart}
    \usepackage{amsfonts}
    \usepackage{amsmath}
    \usepackage{amssymb}
    \usepackage{amscd}
    \usepackage{amstext}
    \usepackage{amsthm}
     \usepackage[usenames,dvipsnames]{color}

 \newtheorem{thm}{Theorem}[section]
    \newtheorem{Lem}[thm]{Lemma}
    \newtheorem{Prop}[thm]{Proposition}
    \newtheorem{Def}[thm]{Definition}
    \newtheorem{Cor}[thm]{Corollary}

    \newtheorem{Remark}[thm]{Remark}
    \newtheorem{NN}[thm]{}

        \newcommand{ \C}{\mathbb C}

    \newcommand{\M}{\mathrm M}
    \newcommand{\tsr}{\mathrm{tsr}\,}
    \newcommand{\csr}{\mathrm{csr}\,}
    
    \newcommand{\Lg}{\mathrm{Lg}\,}
    \newcommand{\diag}{\mathrm{diag}\,}

    \newcommand{\hm}{homomorphism}

    \def\det{\mathrm{det}\,}

   \newcommand{\beq}{\begin{eqnarray}}
\newcommand{\eneq}{\end{eqnarray}}
\newcommand{\CA}{$C^*$--algebra}
\newcommand{\Z}{\mathbb Z}
\newcommand{\R}{\mathbb R}
\newcommand{\andeqn}{\,\,\,{\rm and}\,\,\,}
\newcommand{\rforal}{\,\,\,{\rm for\,\,\,all}\,\,\,}
\newcommand{\tforal}{\,\,\,for\,\,\, all\,\,\,}
\newcommand{\ep}{\epsilon}

\begin{document}
\title{Determinant rank of $C^*$-algebras}
\author{Guihua Gong\qquad Huaxin Lin\qquad Yifeng Xue}
\address{College of Mathematics, Jilin University 130012, China; and \newline Department of Mathematics., University of Puerto Rico, Rio Piedras, PR 00931, USA\\
Research Center for Operator Algebras and Department of Mathematics, Shanghai Key Laboratory of PMMP,\newline
\indent East China Normal University\newline
\indent Shanghai 200062, China }
\maketitle \baselineskip 18pt

\begin{abstract}
Let $A$ be a unital \CA\, and let $U_0(A)$ be the group of unitaries
of $A$ which are path connected to the identity. Denote by  $CU(A)$ the
 closure of the commutator subgroup of $U_0(A).$  Let $i_A^{(1,
n)}\colon U_0(A)/CU(A)\rightarrow U_0(\mathrm M_n(A))/CU(\mathrm
M_n(A))$ be the \hm\, defined by sending $u$ to ${\rm diag}(u,1_n).$
We study the problem when the map $i_A^{(1,n)}$ is an isomorphism
for all $n.$ We show that it is always surjective and is injective
when $A$ has stable rank one. It is also injective when $A$ is a
unital  \CA\, of real rank zero, or $A$ has no tracial state. We
prove that the map is  an isomorphism when $A$ is the
Villadsen's simple AH--algebra of stable rank $k>1.$  We also prove
that the map is an isomorphism for all  Blackadar's unital
projectionless  separable simple \CA s. Let $A=\mathrm M_n(C(X)),$
where $X$ is any compact metric space. It is {noted} that the map
$i_A^{(1, n)}$ is an isomorphism for all $n.$ As a consequence, the
map $i_A^{(1, n)}$ is always an isomorphism for any unital \CA\, $A$
that  is an inductive limit of  finite direct sum of \CA s of the
form $\mathrm M_n(C(X))$ as above. Nevertheless we show that there
are unital \CA s $A$ such that $i_A^{(1,2)}$ is not an isomorphism.
\end{abstract}

\section{Introduction}

 Let $A$ be a unital \CA\, and let $U(A)$ be the unitary group.
 Denote by $U_0(A)$ the normal subgroup which is the connected component
 of $U(A)$ containing the identity of $A.$ Denote by $DU(A)$ the commutator subgroup of $U_0(A)$ and by $CU(A)$ the
 closure of $DU(A).$  We will study the group
 $U_0(A)/CU(A).$
 Recently this group becomes
  an important invariant for the structure of \CA s. It plays an important role in the classification of
  \CA s\, (see \cite{Ell},
  \cite{EG2},\cite{NT},\cite{Th1},\cite{G5},\cite{EGL2},\cite{Lntr1}
  and \cite{GLN}, for example).
 It was shown in \cite{Lntr1} that the map  $ U_0(A)/CU(A)\to
 U_0(\mathrm M_n(A))/CU(\mathrm M_n(A))$ is an isomorphism for all $n\ge 1$ if $A$ is a unital simple \CA\, of tracial rank at most one (see also 3.5 of \cite{Lnhtu}).
  In general, when $A$ has stable rank $k,$
 it was shown by Rieffel (\cite{R2}) that map $U(\mathrm M_k(A))/U_0(\mathrm M_k(A))\to
 U(\mathrm M_{k+m}(A))/U_0(\mathrm M_{k+m}(A))$ is an isomorphism for all integers $m\ge 1.$
 In this case $U(\mathrm M_k(A))/U_0(\mathrm M_k(A))=K_1(A).$  This fact plays an important role
 in the study of the  structure of \CA s, in particular, in the study of
 \CA s of stable rank one since it  simplifies computations when
 $K$--theory involved.
 Therefore it seems natural to ask when the map $i_A^{(1,n)}\colon U_0(A)/CU(A)\rightarrow
 U_0(\mathrm M_n(A))/CU(\mathrm M_n(A))$ is an isomorphism. It will also greatly simplify our understanding  and
 usage of the group  when  $i_A^{(1, n)}$ is an isomorphism for all $n.$
 The main tool to study $U_0(\mathrm M_n(A))/CU(\mathrm M_n(A))$ is the de la Harp and Skandalis determinant as
 studied early by C. Thomsen (\cite{Th}) which involves the tracial state space $T(A)$ of $A.$
 On the other hand, we observe that, when $T(A)=\emptyset,$ $U_0(A)/CU(A)=\{0\}.$ So our attention focuses on the case that $T(A)\not=\emptyset.$  One of the authors
 was asked repeatedly if the map $i_A^{(1, n)}$ is an isomorphism when $A$ has stable rank one.

 It turns out that it is easy to see that the map
 $i_A^{(1, n)}$ is always surjective for all $n.$  Therefore the issue is when $i_A^{(1, n)}$ is injective.  We introduce the following:

 \begin{Def}\label{D2}
Let $ A$ be a unital $C^*$--algebra.
Consider the \hm:
$$
i_A^{(m,n)}\colon U_0( \M_m(A))/CU(\M_m(A))\rightarrow U_0(\M_n(A))/CU(\M_n(A))
$$
{\rm (}induced by $u\mapsto {\rm diag}(u, 1_{n-m})${\rm ) } for integer $n\ge m\ge 1.$ The determinant rank of $A$ is defined to be
$$
\mathrm{Dur}( A)=\min\{m\in\mathbb N\vert\, i_A^{(m,n)}\ \text{is isomorphism for all}\,\,\, n>m \}.
$$

If no such integer exists, we set $\mathrm{Dur}( A)=\infty$.
\end{Def}
We  show that if $A=\lim_{n\to\infty}A_n,$ then
${\rm Dur}(A)\le \sup\limits_{n\ge 1}\{{\rm Dur}(A_n)\}.$
We prove that ${\rm Dur}(A)=1$ for all \CA s of stable rank one
which answers the question mentioned above.
 We also show that ${\rm Dur}(A)=1$ for any unital \CA\, $A$  with real rank zero.
 A closely related and repeated used fact is that  the map
 $u\to u+(1-e)$ is  an isomorphism from $U(eAe)/CU(eAe)$ onto $U(A)/CU(A)$
  when $A$ is a unital simple \CA\, of tracial rank at most one and $e\in A$ is a projection (see 6.7 of \cite{Lntr1} and 3.4 of \cite{Lnhtu}).
  We show in this note that this holds for any simple \CA\, of stable rank one.

  Given Rieffel's early result mentioned above, one might be led to think
  that, when $A$ has higher stable rank, or at least, when $A=C(X)$ for higher
  dimensional finite CW complexes, ${\rm Dur}(A)$ perhaps is large.
  On the other hand it was suggested (see  Section 3 of \cite{Th}) that ${\rm Dur}(A)=1$ may hold
  for most unital simple separable \CA s.
  We found out, somewhat surprisingly, the determinant rank of $\M_n(C(X))$ is always one  for {any} compact metric space $X$
  and for {any}  integer $n\ge 1.$ This, together with previous mentioned result, shows
  that if $A=\lim_{n\to\infty} A_n,$ where $A_n$ is a finite direct sum of \CA s
  of the form $\M_n(C(X)),$ then ${\rm Dur}(A)=1.$
  Furthermore, we found out that ${\rm Dur}(A)=1$ for all Villadsen's examples of unital simple AH--algebras $A$ with higher stable rank.   This research suggests that when $A$ has abundant amount of projections
  then ${\rm Dur}(A)$ is likely one (see part (3) of \ref{T3}).
   In fact, we prove that if $A$ is a unital simple AH--algebra with property  (SP), then ${\rm Dur}(A)=1.$ On the other hand, however, we show
  that if $A$ is a unital projectionless simple  \CA\, and $\rho_A(K_0(A))={\mathbb Z},$ then ${\rm Dur}(A)=1.$  Furthermore, if $A$ is one of the Blackadar's example of unital projectionless simple separable \CA\, with infinite  many extremal tracial states, then ${\rm Dur}(A)=1.$
  Indeed, it looks that it is difficult to find any examples of unital separable simple \CA s whose ${\rm Dur}(A)$
  is larger than one.
  Nevertheless Proposition
  \ref{Tcountex} below provides a necessary condition for ${\rm Dur}(A)=1.$ In fact
  we found that  certain unital separable \CA\, violates this condition, which, in turn, provides
   an example of  unital separable \CA\,  $A$ such that ${\rm Dur}(A)>1.$

  {\bf Acknowledgements}: The most of this work was done
  when second named and third named authors were in the Research Center for Operator Algebras
  in the East China Normal University.  They are both partially supported  by the center.

\section{Preliminaries}

In this section we list some notations and
some basic known facts many of which are taken from \cite{Th} and other
sources for the convenience.

\begin{Def}\label{Dduc}
{\rm
Let $A$ be a \CA. Denote by ${\text M}_n(A)$ the $n\times n$ matrix algebra of over $ A$. If $A$ is not unital, we will use ${\widetilde A}$ for the unitization of $A.$
Suppose that $A$ is unital.
For $u$ in $U_0( A)$, let $[u]$ be the class of $u$ in $U_0( A)/CU(A)$.
}
\end{Def}
We view $ A^n$ as the set of all $n\times 1$ matrices over $ A$. Set
\begin{align*}
S_n( A)&=\{(a_1,\cdots,a_n)^T\in A^n\vert\,\sum\limits^n_{i=1}a_i^*a_i=1\}, \\
\Lg_n( A)&=\{(a_1,\cdots,a_n)^T\in A^n\vert\,\sum\limits^n_{i=1}b_ia_i=1,\
\text{for some}\ b_1,\cdots,b_n\in A\}.
\end{align*}
According to \cite{R1} and \cite{R2}, the topological stable rank, the connected stable
rank
of $ A$ are defined respectively as follows:
\begin{align*}
\tsr( A)=&\min\{\,n\in\mathbb N\vert\,  \Lg_m( A)\ \text{ is dense in}\ A^m,
                     \forall\, m\ge n\,\}\\
\csr( A)=&\min\{\,n\in\mathbb N\vert\, U_0(\M_m( A))\ \text{ acts transitively on}\
        S_m( A),\forall\, m\ge n\,\}.
\end{align*}
If no such integer exists, we set $\tsr(A)=\infty$ and $\csr(A)=\infty$, respectively. Those stable ranks of
$C^*$--algebras are very useful tools in computing $K$--groups of $C^*$--algebras (cf. \cite{R2}, \cite{X0}, \cite{X1}
and \cite{X2} etc.)


\begin{Def}\label{DTA}
{\rm
Let $ A$ be a \CA.
 Denote by $ A_{s.a.}$ (resp. $ A_+$) the set of all self--adjoint (resp. positive)
elements in $ A.$ Denote by  $T(A)$ the tracial state space of $A$.
Let $\tau\in T(A).$
We will also use the notation $\tau$ for the un--normalized trace $\tau\otimes Tr_n$ on $\M_n(A),$
where $Tr_n$ is the  standard  trace for $\M_n(\mathbb C).$
Every tracial state on $\M_n(A)$ has the  form $(1/n)\tau$.
}
\end{Def}

\begin{Def}
{\rm For $a,b\in A$, set $[a,b]=ab-ba$. Furthermore, we set
$$
[ A, A]=\Big\{\sum\limits^n_{j=1}[a_j,b_j]\vert\,a_j,b_j\in A, j=1,\cdots,n, n\ge 1\Big\}.
$$

Now according to \cite{CP}, let $ A_0$ denote the subset of $ A_{s.a.}$ consisting of elements of the form $x-y$,
$x,y\in A_{sa}$ with $x=\sum\limits^{\infty}_{j=1}c_jc_j^*$ and $y=\sum\limits^\infty_{j=1}c_j^*c_j$
(converge in norm) for some sequence
$\{c_j\}$ in $ A$. By \cite{CP}, $ A_0$ is a closed subspace of $ A_{s.a.}$.
}
\end{Def}

The following is  surely known (see \cite{CP} and section 3 of \cite{Th}).

\begin{Prop}\label{P1}
Let $ A$ be a $C^*$--algebra with the unit $1$. The the following statements are equivalent:
\begin{enumerate}
\item[$(1)$] $ A_0= A_{s.a.};$
\item[$(2)$] $1\in A_0;$
\item[$(3)$] $T(A)=\emptyset;$
\item[$(4)$] $ A=\overline{[ A, A]};$
\item[$(5)$] $ A_{s.a.}=\overline{\mathrm{span}\{[a^*,a]\vert\,a\in A\}}$.
\end{enumerate}
\end{Prop}
\begin{proof}(1)$\Rightarrow$(2) is obvious.

(2)$\Rightarrow$(3): If $T( A)\not=\emptyset$, then there is a tracial state $\tau$ on $ A$. Since $1\in A_0$, it follows
that there is a sequence $\{a_j\}$ in $ A$ such that $b=\sum\limits^\infty_{j=1}a_j^*a_j$
and $c=\sum\limits^\infty_{j=1}a_ja_j^*$ are convergent in $ A$ and $1=b-c$. Thus, $\tau(b)=\sum\limits^\infty_{j=1}
\tau(a_j^*a_j)=\tau(c)$ and $\tau(1)=\tau(b-c)=0$. But it is impossible for $\tau(1)=1$.

(3)$\Rightarrow$(1):
This follows from the proof of 3.1 of \cite{Th}.

(4)$\Leftrightarrow$(5):
Let $a,b\in A$ and write $a=a_1+ia_2$ and $ b=b_1+ib_2$, where $a_1,a_2,b_1,b_2\in A_{s.a.}$. Then
\begin{equation}\label{eqn1}
[a,b]=[a_1,b_1]-[a_2,b_2]+i[a_2,b_1]+i[a_1,b_2].
\end{equation}
Put $c_1=a_1+ib_1$, $c_2=a_2+ib_2$, $c_3=a_2+ib_1$ and $c_4=a_1+ib_2$. Then from \eqref{eqn1}, we get that
\begin{equation}\label{eqn2}
[a,b]=\frac{1}{2i}[c_1^*,c_1]-\frac{1}{2i}[c_2^*,c_2]+\frac{1}{\,2\,}[c_3^*,c_3]+\frac{1}{\,2\,}[c_4^*,c_4].
\end{equation}
So by \eqref{eqn2}, (4) and (5) are equivalent.

(5)$\Rightarrow$(1) Let $x\in\mathrm{span}\{[a^*,a]\vert\,a\in A\}$. Then there are elements $a_1,\cdots,a_k\in A$ and
positive numbers $\lambda_1,\cdots,\lambda_k$ such that $x=\sum\limits^j_{i=1}\lambda_i[a^*_i,a_i]-\sum\limits^k_{i=j+1}
\lambda_i[a_i^*,a_i]$ for some $j\in\{1,\cdots,k\}$. Put $c_i=\sqrt{\lambda_i}\,a_i$, $i=1,\cdots,j$ and
$c_i=\sqrt{\lambda_i}\,a_i^*$ when $i=j+1,\cdots,k$. Then $x=\sum\limits^k_{i=1}c_i^*c_i-\sum\limits^k_{i=1}c_ic_i^*
\in A_0$. Since $ A_0$ is closed, we get that
$$
 A_{s.a.}=\overline{\mathrm{span}\{[a^*,a]\vert\,a\in A\}}\subset\overline{ A_0}= A_0\subset A_{s.a.}.
$$

(1)$\Rightarrow$(5) According to definition of $ A_0$, every element $x\in A_0$ has the form $x=x_1-x_2$, where
$x_1=\sum\limits^\infty_{i=1}z_i^*z_i$ and $x_2=\sum\limits^\infty_{i=1}z_iz_i^*$. Thus,
$x\in\overline{\mathrm{span}\{[a^*,a]\vert\,a\in A\}}$ and hence $ A_{s.a.}=\overline{\mathrm{span}\{[a^*,a]\vert\,a\in A\}}$.
\end{proof}

Combining Proposition \ref{P1} with \ref{DTA}, we have
\begin{Cor}\label{C1}
Let $ A$ be a unital $C^*$--algebra with $ A_0= A_{s.a.}$. Then $(\mathrm{M}_n( A))_0=(\mathrm{M}_n( A))_{s.a.}$.
\end{Cor}

Let $a,b\in A_{s.a.}$. Then, for any $n\ge 1$, $\exp(ia)\exp(ib)\big(\exp(-i\frac{a}{\,n\,})\exp(-i\frac{b}{\,n\,})\big)^n
\in DU( A)$ and $\exp(-i(a+b))=\lim\limits_{n\to\infty}\big(\exp(-i\frac{a}{\,n\,})\exp(-i\frac{b}{\,n\,})\big)^n$
by Trotter Product Formula (cf. \cite[Theorem 2.2]{M}). So $\exp(ia)\exp(ib)\exp(-i(a+b))\in CU(A)$.
Consequently,
\begin{equation}\label{xxx}
[\exp(ia)][\exp(ib)]=[\exp(i(a+b))]\quad\text{in}\quad U_0( A)/CU(A).
\end{equation}
The following is  taken from the proof of 3.1 of  \cite{Th}.
\begin{Lem}\label{L1}
Let $a\in A_{s.a.}$
\begin{enumerate}
\item[$(1)$] {If} $a\in A_0$, then $[\exp(ia)]=0$ in $U_0( A)/CU(A);$
\item[$(2)$] If $T( A)\not=\emptyset$ and $\tau(a)=\tau(b)$, $\forall\,\tau\in T( A)$, then
$a-b\in A_0$ and $[\exp(ia)]=[\exp(ib)]$ in $U_0( A)/CU(A)$.
\end{enumerate}
\end{Lem}
%


Combing Lemma \ref{L1} (1) with Corollary \ref{C1}, we have
\begin{Cor}\label{C2}
If $T( A)=\emptyset$, then $U_0(\M_n(A))=CU(\M_n( A))$, $n\ge 1$.
\end{Cor}

\begin{Def}\label{DPU}
{\rm
Let $ A$ be a unital $C^*$--algebra with $T(A)\not=\emptyset$. Let $PU_0^n(A)$ denote the set of all piecewise smooth
maps $\xi\colon [0,1]\rightarrow U_0(\M_n( A))$ with $\xi(0)=1_n$, where $1_n$ is the unit of $\M_n(A)$. For
$\tau\in T(A)$, the de la Harpe and Skandalis function $\Delta^n_\tau$ on
$PU_0^n( A)$ is given by
$$
\Delta^n_\tau(\xi(t))=\frac{1}{2\pi i}\int^1_0\tau(\xi'(t)(\xi(t))^*)\,\text{d}\,t,\quad\forall\,\xi\in PU_0^n( A).
$$
Note we use un--normalized trace $\tau=\tau\otimes Tr_n$ on $\M_n(A).$ This gives a \hm\, $\Delta^n: PU_0^n(A)\to
{\rm Aff}(T(A)).$
}
\end{Def}

We list some of  properties of $\Delta_\tau^n(\cdot)$, which are taken from Lemma 1 and Lemma 3 in \cite{HS}, as following
lemma:

\begin{Lem}\label{L2}
Let $ A$ be a unital $C^*$--algebra with $T( A)\not=\emptyset$. Let $\xi_1,\xi_2,\xi\in PU_0^n(A)$. Then
\begin{enumerate}
\item[$(1)$] $\Delta^n_\tau(\xi_1(t))=\Delta^n_\tau(\xi_2(t))$
for all $\tau\in T(A),$ if $\xi_1(1)=\xi_2(1)$ and  $\xi_1\xi_2^*\in U_0(\widetilde{(C_0(S^1,\M_n(A))});$

\item[$(2)$] there are $y_1,\cdots,y_k\in\M_n(A)_{s.a.}$ such that $\Delta^n_\tau(\xi(t))=\sum\limits^k_{j=1}
\tau(y_j)$, $\forall\,\tau\in T(A)$ and $\xi(1)=\exp(i2\pi y_1)\cdots\exp(i2\pi y_k)$.
\end{enumerate}
\end{Lem}

\begin{Def}\label{Drho}
{\rm
Let $A$ be a $C^*$--algebra with $T( A)\not=\emptyset$. Denote by $\mathrm{Aff}(T( A))$ the set of all real continuous affine functions on
$T(A)$. Define $\rho_ A\colon K_0( A)\rightarrow\mathrm{Aff}(T( A))$ by
$$
\rho_ A([p])(\tau)=\tau(p),\qquad \forall\,\tau\in T( A),
$$
where $p\in\mathrm M_n( A)$  is a projection.

Define $P_n(A)$ the subgroup of $K_0(A)$ which is generated by projections
in $\M_n(A).$  Denote by $\rho_ A^n(K_0(A))$ the subgroup $\rho_ A(P_n(A))$ of $\rho_ A(K_0( A))$. In particular, $\rho_A^1(K_0(A))$ is the subgroup
of $\rho_A(K_0(A))$ which is generated by the image of projections in $A$ under the map $\rho_A.$

}
\end{Def}

\begin{Def}\label{DDert}
{\rm Let $A$ be a unital $C^*$--algebra. Denote by $LU_0^n(A)$ be the set of those piecewise smooth loops
in $U({\widetilde{C_0(S^1, \M_n(A))}}).$
Then by the Bott periodicity, $\Delta^n(LU_0^n(A))\subset \rho_A(K_0(A)).$ Denote by
$$
{\mathfrak{q}}^n: {\rm Aff}(T(A))\to {\rm Aff}(T(A))/\overline{\Delta^n(LU_0^n(A))}
$$
the quotient map. Put $\overline{\Delta}^n={\mathfrak{q}}^n\circ \Delta^n.$
Since $\overline{\Delta}^n$  vanishes on $LU_0^n(A),$ we also
use $\overline{\Delta}^n$ for the \hm\, from $U_0(\M_n(A))$ into
${\rm Aff}(T(A))/\overline{\Delta^n(LU_0^n(A))}.$  An important fact
that we will repeatedly use is that
{\it the kernel of $\overline{\Delta^n}$ is exactly $CU(\M_n(A)),$} by
3.1 of \cite{Th}, a result of Thomsen.
In other words, if $u\in U_0(\M_n(A))$ and $\overline{\Delta^n}(u)=0,$
then $u\in CU(\M_n(A)).$
}
\end{Def}

\begin{Cor}\label{Csingleexp}
Let $A$ be a unital \CA\, and let $u\in U_0(\M_n(A))$ for $n\ge 1.$
Then there is $a\in A_{s.a.}$ and $v\in CU(\M_n(A))$ such that
$
u={\diag}(\exp(i2\pi a), 1_{n-1})v,
$
$($in case that $n=1,$ we make $\diag(\exp(i2\pi a), 1_{n-1})=\exp(i2\pi a)).$

Moreover, if there is a $u\in PU_0^n(A)$ with $u(1)=u,$ we can choose $a$ so that
$\hat{a}=\Delta^n(u(t)),$ where $\hat{a}(\tau)=\tau(a)$ for all $\tau\in T(A).$
\end{Cor}
\begin{proof}
Fix a piecewise smooth path $u(t)\in PU_0^n(A)$  with $u(0)=1$ and $u(1)=u.$
By (2) of \ref{L2}, there are $a_1, a_2,..., a_m\in \M_n(A)_{s.a.}$ such that
$$
u=\prod_{j=1}^m \exp(i2\pi a_j)\andeqn \Delta_\tau^n(u(t))=\tau(\sum_{j=1}^m a_j)\rforal \tau\in T(A).
$$
Put $a_0=\sum_{j=1}^n a_j.$ Write $a_0=(b_{i,j})_{n\times n}.$
Define $a=\sum_{i=1}^n b_{i,i}.$ Then  $a\in A_{s.a.}.$
Moreover,
$$
{\overline{\Delta^n}}(\diag(\exp(-i2\pi a), 1_{n-1})u)=0.
$$
Thus, by 3.1 of \cite{Th}, $\diag(\exp(-2\pi a), 1_{n-1})u\in CU(\M_n(A)).$  Put $v=\diag(\exp(-i2\pi a),1_{n-1})u.$
Then $u=\diag(\exp(i2\pi a), 1_{n-1})v.$
\end{proof}


\section{Determinant rank}
\setcounter{equation}{0}

Let $ A$ be a unital $C^*$--algebra.
Consider the \hm:
$$
\imath_A^{(m,n)}\colon U_0( \M_m(A))/CU(\M_m(A))\rightarrow U_0(\M_n(A))/CU(\M_n(A))
$$
for integer $n\ge m\ge 1.$

We begin with the following:

\begin{Prop}\label{Durp0}
Let $ A$ be a unital $C^*$--algebra with $T( A)\not=\emptyset$.
Then
$$\imath_ A^{(m, n)}\colon U_0( \M_m(A))/CU(\M_m(A))\rightarrow U_0(\M_n( A))/CU(\M_n( A))$$
is surjective for $n\ge m\ge 1$.
\end{Prop}
\begin{proof}
It suffices to show that $\imath_A^{(1, n)}$ is surjective. Let $u\in U_0( \M_n(A)).$  It follows from \ref{Csingleexp}
that $u=\diag(\exp(i2\pi a), 1_{n-1})v$ for some $a\in A_{s.a.}$ and $v\in CU(\M_n(A)).$
Then $\imath_A^{(1,n)}([\exp(i2\pi a)])=[u].$
\end{proof}

\begin{Lem}\label{DurL1}
Let $A$ be a unital \CA\, with $T(A)\not=\emptyset.$ Suppose that $u\in U_0(\M_m(A)).$
\begin{enumerate}
\item[$(1)$]
If $\Delta^n(\diag(u(t),1_{n-m})\in \overline{\Delta^n(LU_0^n(A))}$ for some $n>m,$  where $\{u(t): t\in [0,1]\}$ is a piecewise smooth
path with $u(0)=1_m$ and $u(1)=u,$ then, for any $\ep>0,$ there exist
$a\in \M_m(A)_{s.a.}$ with $\|a\|<\ep,$ $b\in \M_m(A)_{s.a.}$, $v\in CU(\M_m(A))$ and $w\in LU_0^n(A)$
such that
\beq\label{DurL1-0}
u=\exp(i2\pi a)\exp(i2\pi b)v\andeqn \tau(b)=\Delta^n_\tau(w(t))\tforal\tau\in T(A).
\eneq
\item[$(2)$] If $\Delta^m(u(t))\in \overline{\rho_A(K_0(A))}$
for some $u\in PU_0^m(A)$ with $u(1)=u,$ then, for any $\ep>0,$ there exist
$a\in \M_m(A)_{s.a.}$ with $\|a\|<\ep,$ $b\in \M_m(A)_{s.a.}$ and $v\in CU(\M_m(A))$
such that
\beq\label{DurL1-1}
u=\exp(i2\pi a)\exp(i2\pi b) v\andeqn \hat{b}\in \rho_A(K_0(A)),
\eneq
where $\hat{b}(\tau)=\tau(b)$ for all $\tau\in T(A).$
\end{enumerate}
\end{Lem}

\begin{proof}
Let $\ep>0.$
For (1), there is $w\in LU_0^n(A)$ such that
\beq\label{DurL1-n1}
\sup\{|\Delta^n_\tau(u(t))-\Delta^n_\tau(w(t))|:\tau\in T(A)\}<\ep/3\pi
\eneq
There is $a_1\in M_m(A)_{s.a.}$ by Corollary \ref{Csingleexp} such that
\beq\label{DurL1-n2}
\tau(a_1)=\Delta^n_\tau(u(t))-\Delta^n_\tau(w(t))\tforal \tau\in T(A).
\eneq
Combining \eqref{DurL1-n1} with \cite{CP} and the proof of 3.1 of \cite{Th}, we can find $a\in\M_m(A)_{s.a.}$ such that
$\tau(a)=\tau(a_1)$ for all $\tau\in T(A)$
and $\|a\|<\ep/2\pi.$ There is also $b\in A_{s.a.}$ such that
$\tau(b)=-\Delta^n_\tau(w(t))$ for all $\tau\in T(A).$ Put
\beq\label{DurL1-n3}
v(t)=\exp(-i2\pi bt)\exp(-i2\pi at)u(t)\,\,\, {\rm for}\,\,\, t\in [0,1]
\eneq
and $v=v(1).$
Then $\Delta^n(v(t))=0.$ It follows from 3.1 of \cite{Th} that
$v\in CU(A).$ Then $u=\exp(i2\pi a)\exp(i2\pi b)v.$

For (2),  there is an integer $n\ge m$ and projections
$p,q\in \M_n(A)$
such that (for a piecewise smooth path $\{u(t): t\in [0,1]\}$
with $u(0)=1_n$ and $u(1)=u$)
\beq\label{DurL1-2}
\|\Delta_\tau^m(u(t))-\tau(p)+\tau(q)\|<\ep\rforal \tau\in T(A).
\eneq
Let $b\in\M_m(A)_{s.a.}$ such that $\tau(b)=\tau(p)-\tau(q)$ for all $\tau\in T(A)$ (see the proof above) and there is
$a\in\M_m(A)_{s.a.}$ with $\|a\|<\ep$ such that
\beq\label{DurL1-3}
\tau(a)=\Delta_\tau^m(u(t))-\tau(p)+\tau(q)\tforal \tau\in T(A).
\eneq
Now let $v=u\exp(-i2\pi a)\exp( -i2\pi b)$ and set $v(t)=u(t)\exp(-i2\pi at)\exp(-i2\pi bt).$
Then $\Delta^n_\tau(v(t))=0$. It follows from 3.1 of \cite{Th} that $v\in CU(\M_m(A)).$
\end{proof}

Let $A$ be a unital \CA. Let ${\rm Dur}(A)$ be defined as in
\ref{D2}.  It follows from
\ref{C2} that, if $T(A)=\emptyset,$ then ${\mathrm{Dur}}(A)=1.$

\begin{Prop}\label{Durp1}
Let $A$ be a unital \CA. Then,
for any integer $n\ge 1,$
$${\rm Dur}(\M_n(A))\le\Big[{{\rm Dur}(A)-1\over{n}}\Big]+1,$$
where $[x],$ is the integer part of $x,$
\end{Prop}

\begin{proof}
We note that $n( [{{\rm Dur}(A)-1\over{n}}]+1)\ge {\rm Dur}(A).$
\end{proof}

\begin{thm}\label{Lq}
Let $A$ be a unital \CA, $I\subset A$ be a closed ideal of $A$ such that
the quotient map $\pi\colon A\rightarrow A/I$ induces the surjective map from
$K_0(A)$ onto $K_0(A/I).$ Then ${\rm Dur}(A/I)\le {\rm Dur}(A).$
\end{thm}

\begin{proof}
Let $m={\rm Dur}(A)$ and $n>m.$ Let $u\in U_0(\M_m(A/I))$ be such that
$\diag(u, 1_{n-m})\in CU(\M_n(A/I)).$  We will show that $u\in CU(\M_m(A/I)).$

Let $\ep>0.$  By Lemma \ref{DurL1}, without loss of generality, we may assume
that there are $a_1,\, b_1\in(\M_m(A/I))_{s.a.}$ such that
\begin{align}
\notag u=&\exp(i2\pi a_1)\exp(i2\pi b_1)v, ~~ v\in C{U(\M_m(A/I))},\\
\label{Lq1-2}\|a_1\|&<\ep\andeqn \tau(b_1)=\tau(q_1)-\tau(q_2),
\end{align}
where $q_1, q_2\in M_K(A/I)$ are projections for some large $K\ge m,$
for all $\tau\in T(A/I).$  By the assumption, without loss of generality,
we may assume that there are projections $p_1, p_2\in \M_{K}(A)$
such that
$\pi_*([p_1-[p_2])=[q_1]-[q_2]$, where $\pi_*: K_0(A)\to K_0(A/I)$ is induced by $\pi$.
Let $b_2\in(\M_m(A))_{s.a.}$ such that $\tau(b_2)=\tau(p_1)-\tau(p_2)$
for all $\tau\in T(A).$ There is $a\in(\M_m(A))_{s.a.}$ such that
$\pi_m(a)=a_1,$ { where $\pi_m\colon\M_m(A)\rightarrow \M_m(A/I)$ is the induced map induced by $\pi.$} Then, we compute that, by \eqref{Lq1-2},
\begin{equation}\label{Lq1-3}
\pi_m(\exp(i2\pi a))\pi_m(\exp(i2\pi b_2))u^*\in CU(\M_m(A/I)).
\end{equation}
Put $u_1=\pi_m(\exp(i2\pi a))\pi_m(\exp(i2\pi b_2).$
Let $w=\exp(i2\pi b_2).$  Then $\overline{\Delta}(w)=0.$ Since $m={\rm Dur}(A),$ this implies that
$w\in CU(M_m(A)).$ It follows that $\pi_m(w)\in CU(\M_m(A/I))$ which implies (by \eqref{Lq1-3})
that ${\rm dist}(u, \, CU(\M_m(A/I)))<\ep.$
\end{proof}

\begin{thm}\label{DurTlimit}
Let $A=\lim_{n\to\infty}(A_n, \phi_n)$ be a unital \CA, where each $A_n$ is unital.
Suppose that ${\rm Dur}(A_n)\le r$ for all $n.$ Then ${\rm Dur}(A)\le r.$
\end{thm}

\begin{proof}
We will use $\phi_{n_1, n_2}: A_{n_1}\to A_{n_2}$ for $\phi_{n_2}\circ \phi_{n_2-1}\cdots \phi_{n_1}$ and
$\phi_{n_1, \infty}: A_{n_1}\to A$ for the map induced by the inductive limit system.
Let $u\in U_0(\M_r(A))$ such that
$u_1=\diag(u, 1_{n-r})\in CU(\M_n(A))$ for some $n>r.$
Let $\ep>0.$
There is a $v\in DU(\M_n(A))$ such that
\beq\label{DurTlimit-1}
\|u_1-v\|<\ep/8n.
\eneq
Write $v=\prod\limits_{j=1}^K v_j,$ where $v_j=x_jy_jx_j^*y_j$ and $x_j, y_j\in U_0(\M_n(A)),$
$j=1,2,...,K.$
Choose large $N\ge 1$ such that there are $v'\in U_0(\M_r(A_N))$ and $x_j', y_j'\in U_0(\M_n(A_N))$
such that
\beq\label{DurTlimit-2}
\|u-\phi_{N, \infty}(u')\|<\ep/8nK\andeqn \|\phi_{N, \infty}(x_j')-x_j\|<\ep/8nK,\,\,\, j=1,2,...,K.
\eneq
Then, we have by \eqref{DurTlimit-1} and \eqref{DurTlimit-2},
\beq\label{DurTlimit-3}
\|\phi_{N, \infty}(u_1')-\prod_{j=1}^K \phi_{N, \infty}(v_j')\|<\ep/4n,
\eneq
where $u_1'=\diag(u',1_{n-r})$ and $v_j'=x_j'y_j'(x_j')^*(y_j')^*,$ $j=1,2,...,K.$ Then \eqref{DurTlimit-3} implies that
there is $N_1>N$ such that
\beq\label{DurTL-3+}
\|\phi_{N, N_1}(u_1')-\prod_{j=1}^K\phi_{N, N_1}(v_j')\|<\ep/2n.
\eneq
Put $U=\phi_{N, N_1}(u')$ and $U_1=\diag(U,1_{n-r})$ and $w_j=\phi_{N, N_1}(v_j'),$
$j=1,2,...,K.$  Note that $\phi_{N_1, \infty}(U)=\phi_{N, \infty}(u').$
There is $a\in (\M_n(A_{N_1}))_{s.a.}$ by \eqref{DurTL-3+} such that
\beq\label{DurTlimit-4}
U_1=\exp(i2\pi a)\prod_{j=1}^K w_j\andeqn \|a\|<2\arcsin(\ep/8n).
\eneq
There is $b\in (\M_r(A_{N_1}))_{s.a.}$ such that
\beq\label{DurTlimit-5}
\tau(b)=\tau(a)\tforal \tau\in T(A)\andeqn \|b\|<2n\arcsin(\ep/8n).
\eneq
Put $W=\diag(U\exp(-i2\pi b), 1_{n-r}).$ Then $W\in CU(\M_n(A_{N_1})).$
Since ${\rm Dur}(A_{N_1})\le r,$
we conclude that $U\exp(-i 2\pi b)\in CU(\M_r(A_{N_1})).$
It follows that $\phi_{N_1, \infty}(U\exp(-i 2\pi b))\in CU(\M_r(A)).$
However, by \eqref{DurTlimit-1}, \eqref{DurTlimit-2}, \eqref{DurTlimit-5},
\begin{align*}
\|u-\phi_{N_1, \infty}(U\exp(-i 2\pi b))\| &\le \|u-\phi_{N, \infty}(u')\|\\
&\ \ +\|\phi_{N_1, \infty}(U)-\phi_{N_1, \infty}(U\exp(-i 2\pi b))\|\\
&<\ep/8nK+\|1-\exp(-i2\pi \phi_{N_1,\infty}(b))\|\\
&<\ep/8nK+\ep/4<\ep.
\end{align*}
Therefore, ${\rm Dur}(A)\le r$.
\end{proof}

\begin{Prop}\label{T3}
Let $ A$ be a unital $C^*$--algebra with $T( A)\not=\emptyset$. Let $a\in A_{s.a.}$ and put $\hat a(\tau)=\tau(a)$
for all $\tau\in T( A)$.
\begin{enumerate}
\item[$(1)$] If $\exp(2\pi ia)\in CU(A)$, then $\hat a\in\overline{\rho_ A(K_0( A))};$
\item[$(2)$]  If $u\in U_0(A)$ and for some piecewise smooth path
$\{u(t): t\in [0,1]\}$ with $u(0)=1$ and $u(1)=u,$ $\Delta^1(u(t))
\in \overline{\rho_ A^k(K_0( A))}$ for some $k\ge 1$, then
$\diag(u,1_{k-1})\in CU(\M_k(A));$
\item[$(3)$] If $\overline{\rho_ A^1(K_0(A))}=\overline{\rho_ A(K_0(A))}$, then
${\rm Dur}(A)=1.$
\end{enumerate}
\end{Prop}
\begin{proof}
Part (1) follows from \cite{Th}.

(2): By applying Corollary \ref{Csingleexp}, there is $v\in CU(A)$ such that
$$
u=\exp(i2\pi a) v\andeqn  \tau(a)=\Delta_\tau^1(u(t))\tforal \tau\in T(A).
$$
So for any $\ep\in (0,1)$, there are projections $p_1,\cdots,p_{m_1}$, $q_1,\cdots,q_{m_2}\in\mathrm M_k( A)$ such that
\begin{equation}\label{eqn01}
\sup\{|\sum\limits^{m_1}_{j=1}\tau(p_j)-\sum\limits^{m_2}_{j=1}\tau(q_j)-\tau(a)|:\tau\in T( A)\}
<\arcsin(\ep/4)/\pi.
\end{equation}
Set $b=\sum\limits^{m_1}_{j=1}p_j-\sum\limits^{m_2}_{j=1}q_j$ and
$a_0=\diag(a,\overbrace{0,0,...,0}^{(k-1)})$. Then $a_0,b\in
\mathrm M_k( A)_{s.a.}$ and
$$
|{\tau}(a_0)-{\tau}(b)|<\arcsin(\ep/4)/k\pi,\quad \forall\,\tau\in T( {\M_k(A)})$$
by \eqref{eqn01}. Thus, by the proof of Lemma 3.1 in \cite{Th}, we have
$$
\inf\{\|a_0-b-x\|\vert\,x\in(\mathrm M_k(
A))_0\}=\sup\{|{\tau}(a_0-b):\tau\in T({\M_k(A)})\}\le
\arcsin(\ep/4)/k\pi.
$$
Choose $x_0\in (\mathrm M_k( A))_0$ such that $\|a_0-b-x_0\|<2\arcsin(\ep/4)/k\pi$. Put $y_0=a_0-b-x_0$. Then
$\|y_0\|\le 2\arcsin(\ep/4)/{k}\pi.$
Put $u_1=\diag(u,1_{k-1})\exp(-i2\pi y_0).$ Define
$$
w(t)=\diag(u(t), 1_{k-1})\exp(-i2\pi y_0t)\prod_{j=1}^{m_1} \exp(-i2\pi p_jt))(\prod_{j=1}^{m_2}\exp(i2\pi q_jt)
$$
for $t\in [0,1].$  Then $w(0)=1$, $w(1)=u(1)\exp(-i2\pi y_0)=u_1$ and moreover,
\begin{align*}
\Delta^k_\tau(w(t))&=\tau(a)-\tau(y_0) -[\sum_{j=1}^{m_1}\tau(p_j)-\sum_{j=1}^{m_2}\tau(q_j)]\\
&=\tau(a)-\tau(a_0)+\tau(b)-\tau(x_0)-\tau(b)\\
&=\tau(a)-\tau(a_0)=0,\qquad \forall\,\tau\in T(A).
\end{align*}
It follows that $w(1)=u_1\in CU(\M_k(A)).$
Then
$$
\|\diag(u, 1_{k-1})-u_1\|=\|\exp(i2\pi y_0)-1_k\|<\ep.
$$

(3) Let $u\in U_0( A)$ such that $\diag(u,1_{n-1})\in CU(\M_n(A))$.
Let $u(t)$ be a piecewise smooth path with $u(0)=1$ and $u(1)=u.$
Then
$$
\Delta^1(u(t))\in \overline{\rho_A(K_0(A))}=\overline{\rho_ A^1(K_0(A))}.
$$
By part (2), $u\in CU(A).$  This implies that ${\rm Dur}(A)=1.$
\end{proof}

\begin{Prop}\label{Pabelian}
Let $X$ a compact metric space. Then ${\rm Dur}(\M_n(C(X)))=1$, $\forall\,n\ge 1$.
\end{Prop}

\begin{proof}
By Proposition \ref{Durp1}, it suffices to consider the case that $A=C(X).$
One has that
$$
\rho_A^1(K_0(A))=C(X, \mathbb{Z})=\rho_A(K_0(A)).
$$
It follows from part (3) of Theorem \ref{T3} that ${\rm Dur}(A)=1.$
\end{proof}

Combining Theorem \ref{DurTlimit} with Proposition \ref{Pabelian}, we have
\begin{Cor}\label{CAH}
Let $A=\lim_{n\to\infty}(A_n, \phi_n),$ where $A_m=\bigoplus\limits_{j=1}^{m(n)}\M_{k(n,j)}(X_{n,j})$ and
each $X_{n,j}$ is a compact metric space. Then  ${\rm Dur}(A)=1.$
\end{Cor}

\begin{thm}\label{RR0}
Let $A$ be a unital \CA\, with real rank zero. Then $\rho_ A^1(K_0( A))=\rho_ A(K_0( A))$ and ${\rm Dur}(A)=1.$
\end{thm}

\begin{proof}
By \ref{C2}, we may assume that $T(A)\not=\emptyset. $
Since $ A$ is of real rank zero, by \cite[Theorem 3.3]{Zh2},
for any $n\ge 2$ and any non--zero projection $p\in\mathrm M_n( A)$ there are projections $p_1,\cdots,p_n\in A$ such
that $p\sim\diag(p_1,\cdots,p_n)$ in $\mathrm M_n( A)$. Thus, $\tau(p)=\sum\limits^n_{j=1}\tau(p_j)$,
$\forall\,\tau\in T( A)$ and consequently, $\rho_ A^1(K_0( A))=\rho_ A(K_0( A))$.
It follows from the part (3) of Theorem \ref{T3} that  ${\rm Dur}(A)=1.$
\end{proof}

\begin{thm}\label{T1}
Let $ A$ be a unital $C^*$--algebra with $T( A)\not=\emptyset$.
If $\csr(C(S^1, A))\le n+1$ for some $n\ge 1$, then
${\rm Dur}(A)\le n$.
\end{thm}

\begin{proof}
Let $u\in U_0(\M_n(A))$ such that $\diag(u,1_{k})\in CU(\M_{n+k}(A))$ for some integer $k\ge 1.$  Let
$\{u(t): t\in [0,1]\}$ be a piecewise smooth path with $u(0)=1_n$ and $u(1)=u.$
 By \cite{Th},  $\Delta^{n+k}(\diag(u(t), 1_k))\in \overline{\Delta^{n+k}(LU_0^{n+k}(A))}.$
 It follows from the part (1) of Lemma \ref{DurL1} that, for any $\ep>0,$ there are
 $a,\, b\in \M_n(A)_{s.a.}$  and $v\in CU(\M_n(A))$ with $\|a\|<2\arcsin(\ep/4)/\pi$  such that
 \beq\label{T1-n1}
 u=\exp(i2\pi a)\exp(i2\pi b)v\andeqn \tau(b)=\Delta_\tau^{n+k}(w(t))\rforal \tau\in T(A),
 \eneq
 where $w\in LU_0^{n+k}(A).$ Since ${\rm csr}(C(S^1, A))\le n+1,$  then, by \cite[Proposition 2.6]{R2},
 there is $w_1\in LU_0^{n}(A)$ such that
 $\diag(w_1,1_{n+k})$ is homotopy to $w.$ In particular,
 $\Delta_\tau^n(w_1(t))=\Delta_\tau^{n+k}(w(t))$ for all $\tau\in T(A).$
 Consider the piecewise smooth path
 $$
 U(t)=\exp(-2\pi at) \exp(i2\pi bt) w_1^*(t),\quad t\in [0,1].
 $$
 Then $U(0)=1_n$ and $U(1)=\exp(i2\pi b).$ We compute that $\Delta_\tau^n(U(t))=0$, $\forall\,\tau\in T(A)$.
 It follows (by 3.1 of \cite{Th}) that $\exp(i2\pi b)\in CU(\M_n(A)).$ By (\ref{T1-n1}),
 $$
 [u]=[\exp(i2\pi a)]\,\,\,{\rm in}\,\,\,U_0(\M_n(A))/CU(\M_n(A)),
 $$
 Therefore ${\rm dist}(u, CU(\M_n(A)))\le \|\exp(i2\pi a)-1_n\|<\ep.$
 \end{proof}

 \begin{Cor}\label{Cstr1}
 Let $A$ be a unital \CA\, of stable rank one. Then
 ${\rm Dur}(A)=1.$
 \end{Cor}

 \begin{proof}
 This follows from  $\csr(C( S^1, A))\le\tsr( A)+1$ (cf. \cite[Corollary 8.6]{R1})
 and Theorem \ref{T1}.
 \end{proof}

We end this section with the following:

\begin{Prop}\label{Tcountex}
Let $A$ be a unital \CA.
Suppose there is a projection $p\in \M_2(A)$ such  that,
for any $x\in K_0(A)$ with $\rho_A(x)=\rho_A([p]),$
there is no unitary in
$U({\tilde C})$ which represents $x,$
where $C=C_0((0,1), A).$
Then ${\rm Dur}(A)>1.$
\end{Prop}

\begin{proof}
There is $a\in A_+$ such that $\tau(a)=\rho_A([p])(\tau)$ for all $\tau\in T(A).$
Put $u=\exp(i 2\pi a)$ and $v={\rm diag}(u,1).$  Then it follows from (2) of \ref{T3} that
$v\in CU(\M_2(A)).$ This implies that $i_A^{(1,2)}([u])=0.$
Now we will show that $u\not\in CU(A).$
Let $$w(t)=\exp(2 i(1-t) \pi a) \rforal t\in [0,1].$$
Then $w(0)=u$ and $w(1)=1_A.$
If $u\in CU(A),$ then, by 3.1 of \cite{Th}, there is
a continuous and piecewise smooth path of unitaries $\xi\in{\widetilde{C}},$ where
$C=C_0((0,1), A)$ such that
\beq\label{Texam-2}
\Delta_\tau(\xi(t))=\tau(p)\rforal \tau\in T(A).
\eneq
The Bott map shows that the unitary $\xi$ is homotopic to a projection
loop which corresponds to some $x\in K_0(A)$ with $\rho_A(x)=\rho_A([p]),$
which contradicts with the assumption.
\end{proof}

\section{Simple \CA s}
\setcounter{equation}{0}

Let us begin with the following:

\begin{thm}\label{Trr0}
Let $A$ be a unital infinite dimensional simple \CA\, of real rank zero with $T(A)\not=\emptyset$. Then
$$
{\overline{\rho_A^1(K_0(A))}}={\rm Aff}(T(A))\andeqn U_0(A)=CU(A).
$$
\end{thm}

\begin{proof} Let $p\in A$ be a non--zero projection, let $\lambda=n/m$ with $n,m\in\mathbb N$ and let $\ep>0.$
Then by Zhang's half theorem (see Lemma 9.4 of \cite{Lnmem}), there is a projection $e\in A$ such that
$\max\limits_{\tau\in T(A)}|\tau(p)-n\tau(e)|<n\ep/m.$ Thus, $\max\limits_{\tau\in T(A)}|\lambda\tau(p)-m\tau(e)|<\ep$
and consequently, $r\rho_A(p)\in \overline{\rho_A^1(K_0(A))}$, $\forall\,r\in\mathbb R$.

Let $a\in A_{s.a.}.$ Since $A$ has real rank zero, $a$ is a limit of the form $\sum\limits_{j=1}^k\lambda_j p_j,$ where
$p_1, p_2,...,p_k$ are mutually orthogonal projections in $A$ and $\lambda_1,\lambda_2,...,\lambda_k\in\mathbb R.$
Therefore $\hat{a}\in \overline{\rho_A^1(K_0(A))}$ by the above argument, where $\hat{a}(\tau)=\tau(a)$ for all $\tau\in T(A).$
Since ${\rm Aff}(T(A)=\{\hat a\vert\,a\in A_{s.a.}\}$ by \cite[Theorem 9.3]{Lntr1}, it follows from Proposition \ref{RR0}
that
$$
{\rm Aff}(T(A))\subset\overline{\rho_A^1(K_0(A))}=\overline{\rho_A(K_0(A))}\subset {\rm Aff}(T(A)),
$$
that is, ${\rm Aff}(T(A))=\overline{\rho_A^1(K_0(A))}$.

Note that
$$
\rho_A^1(K_0(A))\subset\Delta^1(LU_0^1(A))\subset \rho_A(K_0(A))=\rho_A^1(K_0(A)).
$$
So $\overline{\Delta^1(LU_0^1(A))}=\overline{\rho_A^1(K_0(A))}={\rm Aff}(T(A))$. Therefore ${\overline{\Delta^1}}=0$
(see Definition \ref{DDert}) and the assertion follows.
\end{proof}

For unital simple \CA s, we have the following:

\begin{thm}\label{SimpleT}
Let $ A$ be a unital infinite dimensional simple \CA. Then
${\rm Dur}(A)=1$  if one of the following holds:
\begin{enumerate}
\item[$(1)$] $ A$ is not stably finite;
\item[$(2)$] $ A$ has stable rank one;
\item[$(3)$] $ A$ has real rank zero;
\item[$(4)$] $ A$ is projectionless and $\rho_ A(K_0( A))={\mathbb Z}$ {\rm (}with $\rho_A([1_A])=1${\rm )};
\item[$(5)$] $ A$ has {\rm (SP)} and has a unique tracial state.
\end{enumerate}
\end{thm}

\begin{proof} (1) In this case, there is a non--unitary isometry $u\in\mathrm M_k( A)$ for some $k\ge 2$. Since
$\mathrm M_k( A)$ is also simple, every tracial state on $\mathrm M_k( A)$ is faithful if $T( A)\not=\emptyset$.
This implies that $T( A)=\emptyset$. The assertion follows from Corollary \ref{C2}.

(2) This follows from Corollary \ref{Cstr1}.

(3) This follows from Theorem \ref{Trr0} or Proposition \ref{RR0}.

(4) By the assumption, we have $\rho_ A^1(K_0( A))=\rho_ A(K_0( A))=\mathbb Z$. By Theorem \ref{T3}, ${\rm Dur}(A)=1.$

(5) Let $\ep>0$ and let $\tau\in T( A)$ be the unique tracial state. Let $k\ge 1$  be an integer and  $p\in M_k(A)$ be a
projection. Since $ A$ has (SP),  there is a non--zero projection $q\in A$ such that
$0<\tau(q)<\ep/2$ (see, for example, \cite[Lemma 3.5.7]{Lin}).
Then, there is an integer $m\ge 1$ such that $|m\tau(q)-\tau(p)|<\ep$.
This implies that $\overline{\rho_A^1(K_0(A))}=\overline{\rho_A(K_0(A))}.$
Therefore, by Theorem \ref{T3}, ${\rm Dur}(A)=1.$
\end{proof}

Theorem \ref{SimpleT} indicates that the only cases that $\rm{Dur}(A)$ might not be one for unital simple \CA s are
the cases that $ A$ is stably finite and has stable rank greater than one. The only examples that we know so far that
a unital simple \CA\, is stably finite and has finite stable rank greater than one are the examples given by Villadsen (\cite{V}).

However, we have the following:

\begin{thm}\label{CV}
For each integer $n\ge 1,$
There is a unital simple {\rm AH}--algebras $A$ with $\tsr(A)=n$ such that
${\rm Dur}(A)=1.$
\end{thm}

\begin{proof}
Fix an integer $n>1.$
Let $ A=\lim_{k\to\infty}( A_k, \phi_k)$ be the unital simple AH--algebra with $\tsr(A)=n$ constructed by
Villadsen in \cite{V}. Then $A_1=C({\mathbb D}^n).$ The connecting maps $\phi_k$ are  ``diagonal" maps. More precisely,
$\phi_k(f)=\sum\limits_{j=1}^{n(k)}f(\gamma_{k,j})\otimes p_{k,j}$ for all
$f\in A_k,$ where $p_{k,1}$ is a trivial rank one projection, $A_{k+1}=\phi_k({\rm id}_{A_k})M_{(r(k)}(C(X_k))\phi_k({\rm id}_{A_k})$
(for some large $r(n)$)
for some spaces $X_k$ and $\gamma_{k,j}: X_{k+1}\to X_k$ is a continuous map
(these are $\pi_{i+1}^1$ and some point evaluations as denoted on page 1092 in \cite{V}).
Clearly $A_1$ contains a rank one projection.
Suppose that $A_k,$ as a unital hereditary $C^*$-subalgebra of $M_{r(k)}(C(X_k)),$ contains a rank one projection  $e_k$ (of $M_{r(k)}(C(X_k))$).
Then, since $({\rm id}_{A_k}\circ \gamma_{k,1})\otimes p_{k,1}\le \phi_k({\rm id}_{A_k}),$ $({\rm id}_{A_k}\circ \gamma_{k,1})\otimes p_{k,1}\in A_{k+1}.$
Then $e_k\circ \gamma_{k,1}\otimes p_{k,1}\in A_{k+1}$ which is
a rank one projection.

The above shows every $A_k$ contains a rank one projection.

Now let $p\in\mathrm M_m( A)$ be a projection. We may assume that there is a projection $q\in\mathrm M_m( A_{k_0+1})$
such that $\phi_{k_0+1,\infty}(q)=p.$
Let $e_{k_0}\in  A_{k_0+1}$ be a rank one projection.
Then there is an integer $L\ge 1$  such
that $L \tau(e_{k_0})=\tau(q)$ for all $\tau\in T( A_{k_0+1})$. It follows that
$$
L\tau(\phi_{k_0+1, \infty}(e_{k_0}))=\tau(p)\quad \text{for all}\quad\tau\in T( A).
$$
So $\rho_ A^1(K_0( A))=\rho_ A(K_0( A))$ and hence ${\rm Dur}(A)=1$ by Theorem \ref{T3}.
\end{proof}

\begin{thm}\label{TAH}
Let $ A$ be a unital simple {\rm AH}--algebra with {\rm (SP)} property. Then ${\rm Dur}(A)=1.$
\end{thm}

\begin{proof}
By Theorem \ref{T1} (1), it suffices to show that $i_ A^n$ is injective and {b}y Theorem \ref{T3}, it suffices to show that
$\overline{\rho_ A^1(K_0( A))}=\overline{\rho_ A(K_0( A))}$.

Let $p$ be a projection in $\mathrm M_n( A)$. Since $ A$ is simple, $\inf\{\tau(p)\vert\,\tau\!\in\! T( A)\}\!=\!d\!>\!0.$
Given positive number $\ep<\min\{1/2, d/2\}$. Choose an integer $K\ge 1$ such that $1/K<\ep/2$.
Since $ A$ is a simple unital \CA\, with (SP), it follows from \cite[Lemma 3.5.7]{Lin} that there are mutually orthogonal
and mutually equivalent non--zero projections $p_1, p_2,\cdots,p_K\in  A$ such that $\sum\limits^K_{j=1}p_j\le p$.
We compute that
\begin{equation}\label{eqn11}
\tau(p_1)<\ep/2\andeqn \tau(p_1)<d/K\,\,\,{\rm for\,\,\,all}\,\,\,\tau\in T( A).
\end{equation}
Since $ A$ is simple and unital, there are $x_1, x_2,\cdots,x_N\in  A$ such that
$\sum\limits_{j=1}^N x_j^*p_1x_j=1_ A.$

Write $ A=\lim\limits_{\longrightarrow}( A_m, \phi_m),$ where each $ A_m=\bigoplus\limits_{i=1}^{r(m)}P_{m,j}
\mathrm M_{R(m,j)}(C(X_{m,j}))P_{n,j}$ and $X_{n,j}$ is a connected finite CW--complex and $P_{m,j}\in
\mathrm M_{R(m,j)}(C(X_{m,j}))$ is a projection. Without loss of generality,
we may assume that, there are projections $p_1'\in  A_m,$ $p'\in\mathrm M_n( A_m)$ and elements $y_1,y_2,\cdots,y_N\in A_m$
such that $\phi_{m,\infty}(p_1')=p_1$, $\phi_{m, \infty}(y_j)=x_j$, $(\phi_{m, \infty}\otimes {\rm id}_{\mathrm M_n})(p')=p$
and
\begin{equation}\label{eqn12}
\|\sum\limits^N_{j=1}y_j^*p_1'y_j-1_ A\|<1.
\end{equation}
Write $p_1'$ and $p'$ as
$$
p_1'=p_{1,1}'\oplus p_{1,2}'\oplus\cdots p_{1,r(m)}'\andeqn\\
p'=q_1\oplus q_2\oplus \cdots \oplus q_{r(m)},
$$
here $p_{1,j}'\in P_{m,j}\mathrm M_{R(m,j)}(C(X_{m,j}))P_{m,j}$,
$q_j\in\mathrm M_n(P_{m,j}\mathrm M_{R(m,j)}(C(X_{m,j}))P_{m,j})$, $j=1,\cdots,r(m)$ are projections. Note that \eqref{eqn12}
implies that $p_{1,j}'\not=0,$ $j=1,2,\cdots,r(m).$ Define
$$
r_{1,j}={\rm rank}(p_{1,j}')\andeqn r_j={\rm rank}(q_j),\quad j=1,2,\cdots,r(m).
$$
Then $r_j=l_jr_{1,j}+s_j,$ where $l_j, s_j\ge 0$ are integers and $s_j<r_{1,j}.$
It follows that
\begin{equation}\label{eqn14}
|t(p')-\sum_{j=1}^{r(m)}l_j t(p_{1,j}')|<t(p_1'),\quad\,\forall\,t\in T( A_m)
\end{equation}
Define $q_{1,j}=\phi_{m,\infty}(p_{1,j}'),$ $j=1,\cdots,r(m).$ Then each $q_{1,j}$ is projection in $ A.$ Note that
for each $\tau\in T( A)$, $\tau\circ\phi_{m,\infty}$ is a tracial state on $ A_m$. So by \eqref{eqn14},
$$
|\tau(p)-\sum_{j=1}^{r(m)}l_j\tau(q_{1,j})|<\tau(p_1)<\ep,\quad\forall\,\tau\in T( A).
$$
This implies that $\overline{\rho_ A^1(K_0( A))}=\overline{\rho_ A(K_0( A))}.$
\end{proof}


\begin{Lem}\label{SL1}
Let $ A$ be a unital simple \CA\ with $T( A)\not=\emptyset$, and let $a\in  A_+\setminus \{0\}.$
Then, for any $b\in  A_{s.a.},$ there is $c\in{\mathrm Her}(a)$ such that
$b-c\in A_0.$
\end{Lem}

\begin{proof}
Since $ A$ is simple and unital, there are $x_1,x_2,...,x_m\in A$ such that
$\sum_{j=1}\limits^m x_j^*ax_j=1_ A$. Set $c=\sum\limits_{j=1}^m a^{1/2}x_jbx_j^*a^{1/2}$. Then $c\in{\mathrm Her}(a)$ and
$$
\tau(c)=\sum\limits_{j=1}^m\tau(a^{1/2}x_jbx_j^*a^{1/2})=\sum\limits_{j=1}^m\tau(bx_j^*ax_j)=\tau(b),\quad
\forall\,\tau\in T(A).
$$
It follows from Lemma \ref{L1} (2) that $b-c\in A_0.$
\end{proof}

A special case of the following can be found in 3.4 of \cite{Lnhtu}.

\begin{thm}\label{CornerT}
Let $ A$ be a unital simple \CA\, and let $e\in A$ be a non--zero projection. Consider the map $U_0(eAe)/CU(eAe)\rightarrow U_0(A)/CU(A)$ given by $i_e([u])=[u+(1-e)]$.
Then the map is always surjective and is also injective if $\tsr( A)=1$.
\end{thm}

\begin{proof}
To see $i_e$ is surjective, let $u\in U_0( A).$ Write $u=\prod\limits_{k=1}^n \exp(ia_k)$ for $a_k\in A_{s.a.},$
$k=1,2,...,n.$ By Lemma \ref{SL1}, there are $b_1,...,b_n\in e A e$ such that $b_k-a_k\in A_0.$ Put
$w=e\big(\prod\limits_{k=1}^n\exp(ib_k)\big).$ Then $w\in U_0(e A e).$ Set $v=w+(1-e).$ Then
$v=\prod\limits_{k=1}^n\exp(ib_k)$. Thus, by Lemma \ref{L1} (1),
$$
i_e([w])=[v]=\sum\limits^n_{k=1}[\exp(ib_k)]=\sum\limits^n_{k=1}[\exp(ia_k)]=[u]\quad\text{in}\quad
U_0(A)/CU(A),
$$
that is, $i_e$ is surjective.

To see that $i_e$ is injective when $ A$ has stable rank one, let $w\in U_0(e A e)$ such that $w+(1-e) \in CU(A).$
Since $ A$ is simple, there are $z_1,\cdots,z_n\in A$ such that $1-e=\sum\limits^n_{j=1}z_j^*ez_j$. Put
$X=\begin{bmatrix}ez_1&0&\cdots&0\\ \vdots&\vdots&\ddots&\vdots\\ ez_n&0&\cdots&0\end{bmatrix}\in\mathrm M_n( A)$. Then
\begin{equation}\label{eqn30}
\diag(1-e,\overbrace{0,\cdots,0}^{n-1})=X^*X,\quad XX^*\le\diag(\overbrace{e,e,\cdots,e}^n).
\end{equation}
\eqref{eqn30} indicates that $[1-e]\le n[e]$ in $K_0( A)$. Since $\tsr( A)=1$, we can find
a projection $p\in \mathrm M_s( A)$ for some $s\ge n$ and a unitary $U\in \M_{s+1}( A)$ such that
\begin{equation}\label{eqn31}
\diag(\overbrace{e,\cdots,e}^n,\overbrace{0,\cdots,0}^r)=U{\rm diag}(1-e,p)\,U^*,
\end{equation}
where $r=s-n+1$. Write $v=w+(1-e)$ as $v=\begin{bmatrix}w\\ \ &1-e\end{bmatrix}$ and set
$$
W=\begin{bmatrix}e\\ \ &U\end{bmatrix},\quad Q=\diag(\overbrace{e,\cdots,e}^n,\overbrace{0,\cdots,0}^r).
$$
Then $W\diag(e,1-e,p)(\mathrm M_{s+2}( A))\diag(e,1-e,p)W^*\subset\mathrm M_{n+1}(e A e)\oplus 0$ and
\begin{equation}\label{eqn32}
W\begin{bmatrix}v\\ \ &p\end{bmatrix}W^*=\begin{bmatrix}w\\ \ &U\diag(1-e,p)\,U^*\end{bmatrix}=
\diag(w,Q),
\end{equation}
by \eqref{eqn31}. Note that $\diag(v,p)\in CU(\diag(e,1-e,p)(\mathrm M_{s+2}( A))\diag(e,1-e,p))$.
So by \eqref{eqn32}, $\diag(w,\overbrace{e,\cdots,e}^n)\in CU(\M_{n+1}(e A e))$. Since $\tsr(e A e)=1$,
it follows from Corollary \ref{SimpleT} (2) that $w\in CU(e A e)$.
\end{proof}

 \begin{Lem}\label{SLcom}
 Let $C$ be a non--unital $C^*$--algebra and $B={\widetilde{C}}.$  Assume that $u_1, u_2,\cdots,u_n\in U(\M_k(B))$ for
 some $k\ge 2$. Then, there are unitaries $u_1', u_2',...,u_n'\in \mathrm{M}_k(\widetilde C)$ with $\pi_k(u_j')=1_k$,
 $j=1,\cdots,n$  and  $w,z_j,\bar u_j\in U(\M_k(\mathbb C))$, $j=1,\cdots,n$ such that
$$
 \prod_{j=1}^n u_j=(\prod_{j=1}^nu_j')w,\quad u_j'=z_j^*u_j{\bar u_j}^*z_j,\,\,\,j=1,2,\cdots,n, \text{and}\
 w =\pi_k(\prod_{j=1}^nu_j),
$$
 where $\pi(x+\lambda)=\lambda$, $\forall\,x\in C$ and $\lambda\in\mathbb C$ and $\pi_k$ is the induced homomorphism of
 $\pi$ on $\mathrm{M}_k(B)$.

 Moreover, if $u_j\in U_0(\M_k(B)),$ then we may assume, in addition,
 that each $u_j'\in U_0({\widetilde{\M_k(C)}})$, $j=1,\cdots,n$.
  \end{Lem}
  \begin{proof} Put ${\bar u_j}=\pi_k(u_j)\in U(\M_k(\mathbb C))$. If $n=2,$ then
  \begin{align*}
  u_1u_2&=u_1{\bar u_1}^* ({\bar u_1}u_2{\bar u_1}^*)({\bar u_1}{\bar u_2}^*{\bar u_1}^*)
  ({\bar u_1}{\bar u_2}{\bar u_1}^*{\bar u_1})\\
  &=u_1{\bar u_1}^* ({\bar u_1}u_2{\bar u_1}^*)({\bar u_1}{\bar u_2}^*{\bar u_1}^*)({\bar u_1}{\bar u_2}).
  \end{align*}
  Put $u_1'=u_1{\bar u_1}^*,$  $u_2'={\bar u_1}u_2{\bar u_1}^*{\bar u_1}{\bar u_2}^*{\bar u_1}^*,$
  $w_1={\bar u_1}{\bar u_2},$ $z_1=1_k,$ $z_2={\bar u_1}.$
  Then
  $$
  \pi_k(u_1')=1_k,\ \pi_k(u_2')=\pi_k({\bar u_1}(u_2{\bar u_2}^*){\bar u_1}^*)=1_k\andeqn
  w_1=\pi_k(u_1u_2).
  $$
Thus the lemma holds if $n=2.$ Suppose that the lemma holds for $s.$
  Then
$$
  u_1u_2\cdots u_s u_{s+1}=(u_1'u_2'\cdots u_s')w_s u_{s+1},
$$
where $u_j'\in\mathrm M_k(\widetilde C)$ are unitaries with $\pi_k(u_j')=1_k$, $u_j'=z_j^*u_j{\bar u_j}^*z_j,$
  where $z_j,\bar u_j\in U(\M_k(\mathbb C))$, $j=1,\cdots,s$ and $w_s=\pi_k(\prod\limits_{j=1}^s u_j).$
  It follows that
$$
  \prod_{j=1}^{s+1}u_j=(\prod_{j=1}^su_j')w_su_{s+1}w_s^* (w_s{\bar u_{s+1}}^*w_s^*)
  (w_s{\bar u_{s+1}}).
$$
  Put $u_{s+1}'=w_su_{s+1}w_s^* (w_s{\bar u_{s+1}}^*w_s^*)=w_s(u_{s+1}{\bar u}_{s+1}^*)w_s^*,$
  $z_{s+1}=w_s^*$ and $w_{s+1}=w_s{\bar u_{s+1}}.$
  Then
  \begin{align*}
  \pi_s(u_{s+1}')&=\pi_k(w_s)\pi(u_{s+1}{\bar u_{s+1}}^*)\pi_k(w_s^*)=1_k\andeqn\\
  w_{s+1}&=w_s{\bar u_{s+1}}=\pi_k((\prod_{j=1}^su_j)u_{s+1})=\pi_k(\prod_{j=1}^{s+1}u_j).
  \end{align*}
   The first part of the lemma  follows.

   To see the second part, we first assume that $u_j=\exp(ia_j)$ for some $a_j\in (\mathrm M_k(B))_{s.a.}.$
   Note that ${\bar u_j}=\exp(i {\bar a_j}),$ where ${\bar a}_j=\pi_k(a_j)\in(\mathrm{M}_k(\mathbb C))_{s.a.}$,
   $j=1,\cdots,n$. Consider the path $u_j'(t)= \exp(ita_j)\exp(-it{\bar a_j})$ for $t\in [0,1].$
   Note that, for each $t\in [0,1],$
  $$
   \pi_k(\exp(ita_j)\exp(-it{\bar a_j}))=\exp(i t\pi_k(a_j))\exp(-it\pi_k(a_j))=1_k,\ j=1,\cdots,n.
  $$
    It follows that $u_j'(t)\in\widetilde{\mathrm M_k(\C)}$ for all $t\in [0,1].$
    The case that $u_j=\exp\big(\prod\limits_{k=1}^{m_j} (ia_k)\big)$, $j=1,\cdots,n$ follows from this and what bas been proved.
  \end{proof}

 \begin{Lem}\label{SL3}
 Let $C$ be a non--unital \CA\, and $B={\tilde C}.$ Suppose that
 $z=aba^*b^*,$ where $a,b \in U_0(\M_k(B)).$ Then  $z=yw,$ where  $y\in  CU(\widetilde{\mathrm M_k(C)})$
 with $\pi_k(y)=1_k$ and $w\in CU(\M_k(\mathbb C)).$
 Moreover, if $u=\prod\limits_{j=1}^n z_j,$ where each $z_j\in CU(\M_k(B))$, then $u=yv$,
 where  $y\in CU(\widetilde{\mathrm M_k(C)})$  with $\pi_k(y)=1_k$ and $v\in CU(\M_k(\mathbb C)).$
 \end{Lem}
 \begin{proof}
 Let ${\bar a}=\pi_k(a)$ and ${\bar b}=\pi_k(b).$ Then $\bar a,\bar b\in U(\M_k(\mathbb C))$. It follows from Lemma
 \ref{SLcom}  that there are $a_j, b_j\in U_0(\widetilde{\mathrm M_k(\C)})$ with $\pi_k(a_j)=\pi_k(b_j)=1_k$ and
 $z_j\in U(\M_k(\mathbb C))$, $j=1,2$ such that
 \begin{alignat}{4}
\label{eqn20} ab&=a_1b_1w_1,\quad a_1&=a\bar a^*,\quad b_1&=z_1^*b\bar b^*z_1,\quad &w_1&=\bar a\bar b,\\
\label{eqn21} ba&=b_2a_2w_2,\quad b_2&=b\bar b^*,\quad a_2&=z_2^*a\bar a^*z_2,\quad &w_2&=\bar b\bar a.
\end{alignat}
 Set $x_1=w_1w_2^*z_2^*$ and $x_2=w_1w_2^*z_1$. Then $x_1,x_2\in U_0(\M_k(\mathbb C))$ and
 \begin{align*}
 aba^*b^*&=a_1b_1(w_1w_2^*z_2^*(a{\bar a}^*)z_2w_2w_1^*)(w_1w_2^*(b{\bar b}^*)w_2w_1^*))w_1w_2^*\\
 &=a_1b_1(x_1a^*_1x_1^*)(x_2^*b^*_1x_2) w_1w_2^*
 \end{align*}
by \eqref{eqn20} and \eqref{eqn21}.

Write $a_1=\prod\limits_{j=1}^{m_1}\exp(i y_{1j})$ and $b_1=\prod\limits_{k=1}^{m_2}\exp(i y_{2k})$,
where $y_{1j}, y_{2k}\in(\mathrm M_k(C))_{s.a.}$, $j=1,\cdots,m_1$, $k=1,\cdots,m_2$.
Let $y_{1j}=y_{1j}^+-y_{1j}^-$ and $y_{2k}=y_{2k}^+-y_{2k}^-$ with
$y_{1j}^+, y_{1j}^-,y_{2k}^+,y_{2k}^-\in(\mathrm M_k(C))_+$ for $j=1,\cdots,m_1$ and $k=1,\cdots,m_2$. Set
\begin{alignat*}{3}
c_1&=\sum\limits^{m_1}_{j=1}(y_{1j}^++x_1y_{1j}^-x_1^*)+\sum\limits^{m_2}_{k=1}(y_{2k}^++x_2y_{2k}^-x_2^*),\quad
&d_1&=\sum\limits^{m_1}_{j=1}(y_{1j}^++y_{1j}^-)+\sum\limits^{m_2}_{k=1}(y_{2k}^++y_{2k}^-)\\
c_2&=\sum\limits^{m_1}_{j=1}(y_{1j}^-+x_1y_{1j}^+x_1^*)+\sum\limits^{m_2}_{k=1}(y_{2k}^-+x_2y_{2k}^+x_2^*),\quad
&d_2&=\sum\limits^{m_1}_{j=1}(y_{1j}^-+y_{1j}^+)+\sum\limits^{m_2}_{k=1}(y_{2k}^-+y_{2k}^+).
\end{alignat*}
Then $c_1,c_2,d_1,d_2\in(\mathrm M_2(C))_+$ and clearly, $c_1-d_1,c_2-d_2\in (\mathrm M_k(C))_0$. Therefore,
$(c_1-c_2)-(d_1-d_2)\in(\mathrm M_k(C))_0$. Put $y=a_1b_1(x_1a^*_1x_1^*)(x_2^*b^*_1x_2)$ and $w=w_1w_2^*$.
Then $y\in U_0(\widetilde{\mathrm M_k(C)})$ with $\pi_k(y)=1_k$ and $w=\bar a\bar b\bar a^*\bar b^*\in DU_k(\mathbb C)$.
Moreover, in $U_0(\widetilde{\mathrm M_k(C)})/CU(\widetilde{\mathrm M_k(C)})$,
$$
[y]=[\exp(i(c_1-c_2))]=[\exp(i(d_1-d_2))]=[a_1][b_1][a_1^*][b_1^*]=0.
$$
This proves the first part of the lemma. The second part of the lemma follows.
\end{proof}

\begin{thm}\label{Tstrold}
Let $ A$ be an infinite dimensional unital simple \CA\, with $T( A)\not=\emptyset$ such that, there is $m\ge 1,$
for every hereditary $C^*$--subalgebra  $C,$ ${\rm Dur}({\tilde C})\le m.$ Then ${\rm Dur}(A)=1.$
\end{thm}

\begin{proof}
Let $n\ge 1.$  By Proposition \ref{Durp0}, it suffices to show that $i_A^{(1,n)}$ is injective.
Let $u\in U_0( A)$ with $\diag(u,1_{n-1})\in CU(\M_n(A)).$
 Since $ A$ is simple and infinite dimensional, we can find non--zero mutually orthogonal positive
elements $c_1,...,c_{m}\in A$ and $x_1,...,x_m\in A$ such that
$$
x_j^*x_j=c_1\andeqn x_jx_j^*=c_j,\quad j=2,3,...,m.
$$
Put ${\rm Her}(c_1)=C$ and $B=\widetilde{C}$. Then ${\rm Her}(c_1+c_2+\cdots +c_{m})\cong \mathrm M_{m}(C).$
Note that $\mathrm M_{m}(B)$ is not isomorphic to a  subalgebra of $\mathrm M_m(A)$.

By Lemma \ref{SL1}, we may assume, without loss of generality, that $u=\exp(2\pi ib)$ for some $b\in C_{s.a.}.$
Then by  Theorem \ref{T3} (1), $\hat b\in {\overline{\rho_A(K_0(A))}}.$

Since $A$ is simple and $C$ is $\sigma$--unital, it follows from \cite[Theorem 2.8]{Br} that there is a unitary element
$W$ in $M(A\otimes\mathcal K)$ (the multiplier algebra of $A\otimes\mathcal K)$) such that $W^*(C\otimes {\mathcal K})W
=A\otimes {\mathcal K},$ where $\mathcal K$ is the $C^*$--algebra consisting of all compact operators on $l^2$. Note since $A$ is a unital simple \CA,
 every tracial state $\tau$ on $C$ is the normalization
  of a tracial state restricted on $C.$
Therefore
\begin{equation}\label{eqn22}
\hat b\in\overline{\rho_A(K_0(A))}=\overline{\rho_B(K_0(C))}\subset\overline{\rho_B(K_0(B))}.
\end{equation}
Viewing $b\in B_{s.a},$ consider $v=\exp(i2\pi b)\in U_0(B)$ and $v(t)=\exp(i2\pi tb)$, $t\in [0,1]$. Then \eqref{eqn22}
implies that $\Delta^1(v(t))\in \overline{\rho_B(K_0(B))}.$ By Lemma \ref{DurL1} (2), for any $\ep>0$, there are
$a\in B_{s.a.}$ with $\|a\|<\ep$, $d\in B_{s.a.}$ with $\hat d\in\rho_B(K_0(B))$ and $v_0\in CU(B)$ such that
\begin{equation}\label{eqn23}
v=\exp(i2\pi a)\exp(i2\pi d)v_0.
\end{equation}
Choose projections $p,q\in\M_n(B)$ for some $n>m$ such that
$\tau(\diag(d,0_{(n-1)\times(n-1)}))=\tau(p)-\tau(q)$, $\forall\,\tau\in T(B)$. Thus, $\diag(\exp(i2\pi d),1_{n-1})\in
CU(\M_n(B))$ by Lemma \ref{L1} (2). By the assumption, $i_B^{(m,k)}$ is injective for all $k>m.$  Therefore, we have
$\diag(v,1_{m-1})\in CU(\M_m(B))$ by \eqref{eqn23}.

Let $\ep>0.$ Then there is a $v_1\in DU(\M_{m}(B)),$ such that $\|\diag(v,1_{m-1})-v_1\|<\ep/2.$ We may write
that $v_1=\prod\limits_{j=1}^r z_j$, where $z_j\in\M_{m}(B)$ is a commutator. It follows from Lemma \ref{SL3} that
there are $y\in CU(\widetilde{\mathrm M_m(C)})$  with $\pi_m(y)=1_m$ and $w\in DU(\M_m(\mathbb C))$ such that
$v_1=yw$. Noting that $w=\pi_m(w)=\pi_m(v_1)$ and $\pi(v)=1$, we have $\|1_m-w\|<\ep/2$. Thus
$\|\diag(v,1_{m-1})-y\|<\ep$. Set $v_0=v-1$ and $y_0=y-1_m$. Then
\begin{equation}\label{eqn24}
\diag(v_0,0_{(m-1)\times(m-1)}),y_0\in\mathrm M_m(C)\ \text{and}\ \|\diag(v_0,0_{(m-1)\times(m-1)})-y_0\|<\ep.
\end{equation}
By identifying $1_m+\M_m(C)$ with a unital $C^*$--subalgebra $1_A+{\rm Her}(c_1+c_2+\cdots +c_m)$ of $A,$
we get that $\|\exp(i2\pi b)-y\|<\ep$ by \eqref{eqn24}. Since $y\in CU({\widetilde{\M_m(C)}})\subset CU(A)$ and hence
$u\in CU(A)$, that is, ${\rm Dur}(A)=1$.
\end{proof}

\begin{Cor}\label{Cher}
Let $A$ be a unital simple \CA. Suppose that, there is an integer $K\ge 1$ such that
${\rm csr}(C(S^1, C))\le K$ for every hereditary $C^*$--subalgebra $C.$ Then
${\rm Dur}(A)=1.$
\end{Cor}
\begin{proof}
It follows from Theorem \ref{T1} that ${\rm Dur}({\tilde C})\le \max\{K-1,1\}.$ Then Theorem \ref{Tstrold} applies.
\end{proof}

\begin{Def}\label{Ddrho}
{\rm
Let $A$ be a \CA\, with $T(A)\not=\emptyset$.
Define
\begin{align*}
D(\rho_A^1(K_0(A)), \rho_A(K_0(A)))&=\sup\{ {\rm dist}(x,\rho_A^1(K_0(A)))\vert\, x\in\overline{\rho_A(K_0(A))}\}\\
&=\sup\{ {\rm dist}(x,\rho_A^1(K_0(A)))\vert\, x\in \rho_A(K_0(A))\}.
\end{align*}
}
\end{Def}

\begin{thm}\label{distrho}
Let $A$ be a unital simple \CA\, with $T(A)\not=\emptyset$ such that there is $M>0$ such that
$D(\rho_{C}^1(K_0(C)), \rho_{C}(K_0(C)))<M$ for all non-zero hereditary $C^*$--subalgebra $C$ of $A.$
 Then ${\rm Dur}(A)=1.$
\end{thm}

\begin{proof}
Let $u\in U_0(A)$ such that $\diag(u,1_{n-1})\in CU(\M_n(A))$. By Corollary \ref{Csingleexp}, we may assume that
$u=\exp(i2\pi a)$ for some $a\in A_{s.a.}$. Then $\hat{a}\in \overline{\rho_A(K_0(A))}$ by Theorem \ref{T3} (1).

Given $\ep>0$. Choose an integer $N\ge 1$ such that $M/N<\ep/2\pi.$ There are mutually orthogonal non--zero positive
elements $c_1,c_2,...,c_N$ in $A$ and elements $x_1,x_2,...,x_N\in A$ such that
\beq\label{dist-6}
x_j^*x_j=c_1\andeqn x_jx_j^*=c_j,\,\,\, j=2,3,...,N.
\eneq
Let $C={\rm Her}(c_1)$ and $B=\widetilde C$. It follows from \ref{SL1} that there is $b\in C_{s.a.}$ such that $a-b$ in $A_0,$ i.e.,
$\tau(a)=\tau(b)$ for all $\tau\in T(A).$ Therefore $[\exp(i2\pi a)]=[\exp(i2\pi b)]$ in $U_0(A)/CU(A)$
by Lemma \ref{L1} (2).

Since $A$ is a unital simple \CA\, and $C$ is $\sigma$--unital, it follows from the proof of
Theorem \ref{Tstrold} that $\rho_C(b)\in \overline{\rho_C(K_0(C))}.$
Therefore, by the assumption,
there are projections $p_1,p_2,...,p_{k_1}, q_1,q_2,...,q_{k_2}\in C$
 such that
$$
\sup_{\tau\in T(C)}|\tau(b)-\big(\sum_{i=1}^{k_1} \tau(p_i)-\sum_{j=1}^{k_2}\tau(q_j)\big)|< M.
$$
Put $d=\sum\limits_{i=1}^{k_1}p_i-\sum\limits_{j=1}^{k_2}q_j$ and $f=b-d.$ Then $\exp(i2\pi d)\in CU(A)$ by \eqref{xxx}
and $[\exp(i2\pi f)]=[\exp(i2\pi b]$ in $U_0(A)/CU(A)$.
Moreover, from
$$
\inf\{\|f-x\|\vert\,x\in C_0\}=\sup\{|\tau(f)|\vert\,\tau\in T(C)\}<M
$$
(see the proof of 3.1 of \cite{Th}), there is $f_0\in C_0$ and $f_1\in C_{s.a.}$ with $\|f_1\|<M$
such that $f=f_1+f_0.$ By Lemma \ref{L1} (1), $\exp(i2\pi f_0)\in CU(A).$ Since $f_1\in C_{s.a.},$ by \eqref{dist-6},
there are $g_i\in {\rm Her}(c_i)$ with
\beq\label{dist-10}
\|g_i\|\le \|f_1\|/N\andeqn \tau(g_i)=\tau(f_1/N)\rforal \tau\in T(A),
\eneq
$i=1,2,...,N.$ Put $g=\sum\limits_{i=1}^n g_i\in A.$ Then, by \eqref{dist-10},
\beq\label{dist-11}
\|\exp(i2\pi g)-1_A\|<M/N<\ep\andeqn \overline{\Delta^1}(\exp(i2\pi f)\exp(-i2\pi g))=0.
\eneq
So $\exp(i2\pi f)\exp(-i2\pi g)\in CU(A)$ and consequently, ${\rm dist}(e^{i2\pi a}, CU(A))<\ep.$
\end{proof}

Bruce Blackadar in \cite{B1} constructed three  examples
of unital simple separable nuclear \CA s $A,$ $A_\triangle,$ $A_H,$ with no non--trivial projections.
By 4.9 of \cite{B1}, $K_0(A)={\mathbb Z}$ and with a unique tracial state.  It follows from  (4) of
Corollary \ref{SimpleT} that ${\rm Dur}(A)=1.$ We turn to his examples
$A_\triangle$  and $A_H$ which may have rich tracial spaces. it should be
also noted,
$\M_2(A_\triangle)$ has a projection $p$ with $\tau(p)=1/2$ for all $\tau\in T(A_\triangle)$.  In particular, this implies that
$$
\overline{\rho_{A_\triangle}^1(K_0(A_\triangle))}\not= \rho_{A_\triangle}(K_0(A)).
$$
However, ${\rm Dur}(A_\triangle)=1$ as shown below. It follows that
there is a unitary $u\in {\tilde C},$ where $C=C_0((0,1), A),$
which represents a projection $q$ with $\tau(q)=1/2$ for all $\tau\in T(A_\triangle).$

\begin{Prop}\label{C4}
Let $B$ be a unital AF--algebra and $\sigma$ be an automorphism on $B$. Put $M_\sigma=\{f\in C([0,1],B)\,\vert\,
f(1)=\sigma(f(0))\}$. Then ${\rm Dur}(M_\sigma)=1.$
\end{Prop}
\begin{proof}Clearly, $T(M_\sigma)\not=\emptyset$. From the exact sequence of $C^*$--algebras
$$
0\longrightarrow C_0((0,1),B)\longrightarrow M_\sigma\longrightarrow B\longrightarrow 0,
$$
we obtain the exact sequence of $C^*$--algebras as follows:
\begin{equation}\label{eqn10}
0\longrightarrow C_0((0,1)\times S^1,B)\longrightarrow C( S^1,M_\alpha)\longrightarrow C( S^1,B)
\longrightarrow 0.
\end{equation}
Since $B$ is an AF--algebra, it follows from \cite[Corollary 2.11]{N1} that
$$
\csr(C( S^1,B))=\csr(C( S^1))=2,\quad \csr(C_0((0,1)\times S^1,B))=
\csr(C_0((0,1)\times S^1))=2
$$
and consequently, applying \cite[Lemma 2]{N} to \eqref{eqn10}, we get that
$$
\csr(C( S^1,M_\sigma))\le\max\{\csr(C( S^1,B)),\csr(C_0((0,1)\times S^1,B))\}\le 2.
$$
Therefore ${\rm Dur}(A)=1$ by Theorem \ref{T1}.
\end{proof}

\begin{Cor}\label{Cblack}
${\rm Dur}(A_\triangle)=1$ and ${\rm Dur}(A_H)=1.$
\end{Cor}

\begin{proof}
Both \CA s are of the form $\lim_{n{\to}\infty} A_n,$ where
each $A_n\cong M_\sigma,$ where $M_\sigma$ is as in Corollary \ref{C4}.
As in Corollary \ref{C4}, ${\rm Dur}(A_n)=1.$ By Theorem \ref{DurTlimit}, ${\rm Dur}(A_\triangle)=1$ and ${\rm Dur}(A_H)=1.$
\end{proof}



\section{\CA s with $\mathrm{Dur}(A)$>1}

In this section, we will present a unital \CA\, $C$ such that
${\rm Dur}(C)=2.$ In particular, we will show that there are \CA s
which satisfy the condition described in \ref{Tcountex}.

\begin{NN}\label{5R1}
{\rm  We first list some standard facts from elementary topology. We will give a brief proof for each fact for the reader's convenience.
}
\end{NN}

\underline{Fact 1}: Let
$$B_d(0)=\{(x_1,x_2,x_3,x_4)\in {\mathbb R}^4\,|\,\sqrt{x_1^2+x_2^2+x_3^2+x_4^2}\le d\}.$$
Let $f: B_d(0)\times S^1\to S^3=SU(2)$ be a continuous map which is not surjective. Then there is a homotopy
$$F: B_d(0)\times S^1\times[0,1]\to S^3=SU(2)$$
such that $F(x,e^{i\theta},0)=f(x,e^{i\theta})$, $F(x,e^{i\theta},s)=f(x,e^{i\theta})$ if $\|x\|=d$ (in other words $x\in \partial B_d(0)$) and $g(x, e^{i\theta})=F(x, e^{i\theta},1)$ satisfies
$$g(0,e^{i\theta})=F(0,e^{i\theta},1)=\begin{bmatrix}1&0\\0&1\end{bmatrix}\in SU(2)=S^3.$$

\begin{proof}
Assume $f$ misses a point $z\in S^3=SU(2)$ and that $z\not=\begin{bmatrix}1&0\\0&1\end{bmatrix}\in SU(2)$. Then $S^3\setminus\{z\}$ is homeomorphic to $D^3=\{(x,y,z)\,|\,x^2+y^2+z^2<1\}$ with the identity matrix mapping to $(0,0,0)$. Without loss of generality, we can assume that $f$ is a map from $B_d(0)\times S^1$ to $D^3$. Let $F:B_d(0)\times S^1\times[0,1]\to D^3$ be defined by
$$F(x,e^{i\theta},s)=f(x,e^{i\theta})\max\{1-s,\|x\|/d\},$$
which satisfies the condition.
\end{proof}

\underline{Fact 2}: Let $f,g:S^4\times S^1\to SU(n)\subset U(n)$ (where $n\geq2$) be continuous maps. If $f$ is homotopic to $g$ in $U(n)$, then they are homotopic in $SU(n)$ also.

This follows from the fact that there is a continuous map $\pi: U(n)\to SU(n)$ with $\pi\circ \imath={\rm id}|_{SU(n)}$, where $\imath :SU(n)\to U(n)$ is the inclusion.

\underline{Fact 3}: Let $\xi\in S^4$ be the North pole. Suppose that $f,g:S^4\times S^1\to SU(n)$ are two continuous maps such that
$$f(\xi,e^{i\theta})=1_n=g(\xi,e^{i\theta})$$
for all $e^{i\theta}\in S^1$. If $f$ and $g$  are homotopic in $SU(n)$, then there is a homotopy
$$F: S^4\times S^1\times[0,1]\to SU(n)$$
such that $F(x,e^{i\theta},0)=f(x,e^{i\theta})$, $F(x,e^{i\theta},1)=g(x,e^{i\theta})$ for all $x\in S^4$, $e^{i\theta}\in S^1$ and
$F(\xi,e^{i\theta},t)=1_n$ for all $e^{i\theta}\in S^1$.

\begin{proof}
 Let $G:S^4\times S^1\times[0,1]\to SU(n)$ be a homotopy between $f$ and $g$. That is $G(\cdot, \cdot,0)=f$ and
 $G(\cdot,\,\cdot,\,1)=g$. Let $F:S^4\times S^1\times[0,1]\to SU(n)$ be defined by
$$F(x,e^{i\theta},t)=G(x,e^{i\theta},t)(G(\xi,e^{i\theta},t))^*.$$
Then $F$ satisfies the condition.
\end{proof}

\begin{NN} \label{5R3}
{\rm
We will describe the projection $P\in \M_4(C(S^4))$ of rank $2$, which represents the class of $(2,1)\in\Z\oplus\Z\cong K_0(C(S^4))$ as follows: one can regard $S^4$ as the quotient space $D^4/\partial D^4$, where
$$D^4=\{(z,w)\in\C^2\,|\, |z|^2+|w|^2\leq1\}.$$
It is standard to construct a unitary
$$\alpha: D^4\to U_4(\C)=U(\M_4(\C))$$
such that $\alpha(0)=1_4$ and for any $(z,w)\in\partial D^4$ (that is $ |z|^2+|w|^2=1$)
$$\alpha(z,w)\,{\stackrel{\Delta}=}
\begin{bmatrix}z&w&0&0\\-\bar{w}&\bar{z}&0&0\\0&0&\bar{z}&-w\\0&0&\bar{w}&z\end{bmatrix}\,{\stackrel{\Delta}=}\,\begin{bmatrix}\beta(z,w)&0\\0&\beta(z,w)^*\end{bmatrix},$$
where $\beta(z,w)=\begin{bmatrix}z&w\\-\bar{w}&\bar{z}\end{bmatrix}$  for $(z,w)\in\partial D^4=S^3$, represented the generator of $K_1(C(S^3))$. $P: S^4\to U_4(\C)$ is defined by
$$ P(z,w)\,{\stackrel{\Delta}=}\,\alpha(z,w)\begin{bmatrix}1_2&0_2\\0_2&0_2\end{bmatrix}\alpha^*(z,w)$$
Note that $\alpha$ is not defined as a function from $S^4=D^4/\partial D^4$ to $U(4)$, but $P$ is so defined, since
$$
P(z,w)=\begin{bmatrix}1_2&0_2\\0_2&0_2\end{bmatrix}\quad\forall\,(z,w)\in \partial D^4
$$
and $\partial D^4$ is identified with the North pole $\xi\in S^4$. Hence $P(\xi)=\begin{bmatrix}1_2&0_2\\0_2&0_2\end{bmatrix}$.
}
\end{NN}

\begin{NN}
{\rm For a compact metric space $X$ with a given base point and a $C^*$algebra $A$, in the rest of the paper,  denoted by $C_0(X,A)$ ($C_0(X, \C)$ will be simplified as $C_0(X)$), we mean  the $C^*$ algebra of the continuous function from $X$ to $A$ which vanishes at the base point. (Most spaces we used here have obvious base point, which we will not mention afterward.)
Let $A=C_0(S^1,P\M_4C(S^4)P)$. Let $\tilde{A}$ be the unitization of $A$. Let $B=C_0(S^1,C(S^4))$. Since $A$ is a corner of $\M_4(B)$ and $B$ is a corner of
$\M_2(A)$ (note a trivial projection of rank $1$ is equivalent to a sub projection of $P\oplus P$), $A$ is stably
isomorphic to $B$. Let $\tilde{B}$ be a unitization of $B$. Then $\tilde{B}=C(S^4\times S^1)$ and
$$
K_1(\tilde{A})\cong K_1(A)\cong K_1(B)\cong K_1(\tilde{B})\cong\Z\oplus\Z.
$$
}
\end{NN}

\begin{NN}
{\rm
 For any unitary $u\in M_4(C(S^4\times S^1))$, in the identification of $[u]\in K_1(C(S^4\times S^1))$ with $\Z\oplus\Z$, the first component corresponding to the winding number of
$$S^1\hookrightarrow S^4\times S^1\stackrel{\det u}\longrightarrow S^1\subset \C$$
that is, the winding number of the map
$$e^{i\theta}\to\text{determinant }u(\xi, e^{i\theta}),$$
where $\xi$ is the North pole of $S^4$. Hence if $u: S^4\times S^1\to SU(n)$, then the first component of $[u]\in K_1(C(S^4\times S^1))\cong \Z\oplus\Z$ is automatically zero.
}
\end{NN}

\begin{Lem}\label{5L55}
  Let $u:S^4\times S^1\to SU(2)$. Then $u\in \M_2(C(S^4\times S^1))$ represents the zero element in
$K_1(C(S^4\times S^1))$. In other words, if $u\in SU_n(S^4\times S^1)$ represents a non-zero element in $K$--theory, then
$n\geq 3$.
\end{Lem}

\begin{proof}
Let $f: S^4\times S^1\to S^5$ be the standard quotient map by identifying $\{\xi\}\times S^1\cup S^4\times \{1\}$ into a single point. Consider $u: S^4\times S^1\to SU(2)$. Without loss of generality, assume $u(\xi,1)=1_2\in SU(2)$. Then $u|_{S^4\times \{1\}}: S^4\to SU(2)=S^3$ represents an element in $\pi_4(S^3)\cong\Z/2\Z$. Therefore $u^2|_{S^4\times \{1\}}: S^4\to SU(2)=S^3$ is homotopically trivial, with $(\xi,1)\in S^4\times S^1$ as a fixed point. Evidently, $u^2|_{ \{\xi\}\times S^1}:S^1\to S^3=SU(2)$ is homotopically trivial with $(\xi,1)\in S^4\times S^1$ as a fixed point. Consequently
$$u^2|_{S^4\times \{1\}\cup\{\xi\}\times S^1}:S^4\times\{1\}\cup\{\xi\}\times S^1\to S^3$$
is homotopically trivial with $(\xi,1)\in S^4\times S^1$ as a fixed base point. There is a homotopy
$$F: (S^4\times\{1\}\cup\{\xi\}\times S^1)\times [0,1]\to S^3$$
with $F(\bullet~,0)=u^2|_{S^4\times \{1\}\cup\{\xi\}\times S^1}$ and
$$F(x,1)=1_2\quad\forall x\in S^4\times\{1\}\cup\{\xi\}\times S^1.$$
The following is a well--known easy fact:

For any relative CW complex $(X,Y)$ ($Y\subset X$), any continuous map from $Y\times I\cup X\times\{0\}\to Z$ (where $Z$ is any other CW complex) can be extended to a continuous map $X\times I\to Z$.

Hence, there is a homotopy $G: (S^4\times S^1)\times[0,1]\to S^3$ with $G(\bullet~,0)=u^2$, and $G|_{S^4\times\{1\}\cup\{\xi\}\times S^1\times[0,1]}=F$. Let $v:S^4\times S^1\to SU(2)$ be defined by $v(x)=G(x,1)$, then $[v]=[u^2]\in K_1(C(S^4\times S^1))$ and $v$ maps $S^4\times\{1\}\cup\{\xi\}\times S^1$ to $1_2\in SU(2)$. Consequently, $v$ passes to a map
$$v_1:S^5\,{\stackrel{\Delta}=}\,S^4\times S^1/S^4\times\{1\}\cup\{\xi\}\times S^1\to S^3=SU(2)$$
and represents an element in $\pi_5(S^3)=\Z/2\Z$. Hence $v_1^2:S^5\to S^3$ is a homotopically trivial and therefore $v^2$ is homotopically trivial. So we have
$$4[u]=2[u^2]=2[v]=[v^2]=0\in K_1(C(S^4\times S^1))$$
which implies $[u]=0\in K_1(C(S^4\times S^1))$.
\end{proof}

\begin{Remark}\label{R56}
{\rm
In the proof of \ref{5L55}, we in fact proved the following fact:
For any $u: S^4\times S^1\to SU(2),$ the map $u^4: S^4\times S^1\to SU(2)$ is homotopically
trivial.

}
\end{Remark}

\begin{NN}\label{R57}
{\rm
Note that $P\in \M_4(C(S^4))$ can be regarded as a projection in $\M_4(C(S^4\times S^1)),$ still denote by $P,$ i.e.,
for fixed $x\in S^4,$ $P(x, \, \cdot)$ is a constant projection along the direction $S^1.$
Then
\beq\label{R57-1}
K_1(A)\cong K_1({\tilde A})\cong K_1(C(S^4\times S^1))\cong K_1(P\M_4(C(S^4\times S^1))P),
\eneq
where $A=C_0(S^1, P\M_4(C(S^4))P)$ is defined in \ref{5R3}.
Let
\beq\label{5r56-1}\nonumber
E=\{(\zeta, u): \zeta\in S^4\times S^1, u\in \M_4({\mathbb C})\,\,\,{\rm with}\,\,\,
P(x)uP(x)=u\andeqn u^*u=uu^*=P(x)\}\\\nonumber
\andeqn SE=\{(\zeta, u)\in E: {\rm det}(P(x)uP(x)+(1-P(x))=1\}.
\eneq
Then $E\to S^4\times S^1$ (and $SE\to S^4\times S^1,$ respectively)
is a fiber bundle with the fiber being $U(2)$ (or $SU(2),$ respectively).
Also the unitaries in $P\M_4(C(S^4\times S^1))P$ is one to one corresponding to the cross
sections of bundle $E\to S^4\times S^1.$
For this reason, we will call a cross section of bundle $SE\to S^4\times S^1$ a unitary
(of $P\M_4(C(S^4\times S^1))P$) with determinant one everywhere.
}
\end{NN}

\begin{thm}\label{5T58}
If $u\in P\M_4(C(S^4\times S^1))P$ has determinant one everywhere, that is $u$ is a cross section of
$SE\to S^4\times S^1,$ then $[u]=0$ in $K_1(P\M_4(C(S^4\times S^1))P).$
\end{thm}

\begin{proof}
Note that $SE\to S^4\times S^1$ is smooth fiber bundle over the smooth manifold $S^4\times S^1.$ By a standard result in differential topology, $u$ is homotopic to a $C^{\infty}$-section.
Without loss of generality, we may assume that $u$ itself is smooth. Identify the North pole
$\xi\in S^4$ with $0\in {\mathbb R}^4$ and a neighborhood of $\xi$ with $B_{\ep}(0)\subset {\mathbb R}^4$ for
$\ep>0.$
Since $B_{\ep}(0)$ is contractible, $SE|_{B_\ep(0)\times S^1}$ is a trivial bundle.
Note that the projection $P\in \M_4(C(S^4\times S^1))$ is constant along $S^1,$ hence
$
SE\cong SE|_{S^4\times \{1\}}\times S^1
$
and $SE|_{B_\ep(0)\times S^1}\cong SE|_{B_\ep(0)\times \{1\}}\times S^1,$
in other words, the fiber is constant along $S^1$  and
$SE|_{B_\ep(0)\times \{1\}}$ is trivial and isomorphic to $(B_\ep(0)\times \{1\})\times SU(2).$)
There is a smooth bundle isomorphism
\beq\label{5T58-3}
\gamma: SE|_{B_\ep(0)\times S^1}\to (B_\ep(0)\times S^1)\times SU(2).
\eneq
Then
$$
\gamma \circ u |_{B_\ep(0)\times S^1}: B_\ep(0)\times S^1\to (B_\ep(0)\times S^1)\times SU(2)
$$
is smooth map with
$$
\pi_1\circ (\gamma\circ u)|_{B_\ep(0)\times S^1}={\rm id}_{B_\ep(0)\times S^1},
$$
where $\pi_1: (B_\ep(0)\times S^1)\times SU(2)\to B_\ep(0)\times S^1$ is the projection
onto  the first coordinate.
Denote $\phi=\pi_2\circ (\gamma\circ u|_{B_\ep(0)\times S^1}),$
where $\pi_2: (B_\ep(0)\times S^1)\times SU(2)\to SU(2)$ is the projection onto the second
coordinate. Since $\phi$ is smooth, $\phi|_{\{\xi\}\times S^1}$ is not onto
$SU(2)$ (note ${\rm dim}(SU(2))=3$ and
${\rm dim}(S^1)=1,$ so it cannot be onto).
Therefore, if $\ep$ is small enough, $\phi|_{B_\ep(0)\times S^1}$ is not onto.
By Fact 1 of \ref{5R1}, $\phi$ is homotopic to a constant map $\phi_1: B_\ep(0)\times S^1\to SU(2)$
with
\beq\label{5T58-4}
\phi_1(\{\xi\}\times S^1)=\begin{bmatrix} 1 & 0\\
                                               0 & 1\end{bmatrix}\andeqn \phi |_{\partial B_\ep(0) \times S^1}=\phi_1|_{\partial B_\ep(0)\times S^1}
                                     \eneq
                                     via a homotopy
                                     $F: (B_\ep(0)\times S^1)\times [0,1]\to SU(2)$
                                     with $F(x, e^{i\theta}, t)$ is constant with respect to $t$ if $x\in \partial{B_\ep(0)}.$

Let $u_1: B_{\epsilon}(0)\times S^1\to SE$ be the cross section defined by
$$u_1(x,e^{i\theta})=\gamma^{-1}(((x,e^{i\theta})),\phi_1(x,e^{i\theta}))\in SE.$$
Then $u_1(x,e^{i\theta})=u(x,e^{i\theta})$ if $x\in \partial B_{\epsilon}(0).$
We can extend $u_1$ to $S^4\times S^1$ by defining
$$u_1(x,e^{i\theta})=u(x,e^{i\theta})\quad\text{if}\, (x,e^{i\theta})\notin B_{\epsilon}(0)\times S^1.$$
Hence $u_1$ is a section of $SE$ with
$$u_1(\xi,e^{i\theta})=\begin{bmatrix}1_2&0_2\\0_2&0_2\end{bmatrix}= P(\xi),\,\rforal e^{i\theta}\in S^1.$$
Furthermore $u_1$ is homotopic to $u$ by a homotopy which is constant homotopy on $(S^4\setminus B_{\epsilon}(0))\times S^1$ (on which $u_1=u$) and agrees with  $F$ on $B_{\epsilon}(0)\times S^1$. Hence $[u]=[u_1]\in K_1(PM_4(C(S^4\times S^1))P).$
Recall $S^4$ is obtained from $D^4=\{(z,w)\in\C^2\,|\, |z|^2+|w|^2\leq1\}$ by identifying $\partial D^4=\{(z,w)\in {\mathbb C}^2\,|\, |z|^2+|w|^2=1\}$ with the North pole $\xi\in S^4$. Recall $P\in M_4(C(S^4))$ (regarded as in $M_4(C(S^4\times S^1))$ which is  a constant along the direction of $S^1$) is defined as
$$
P(z,w)=\alpha(z,w)\begin{bmatrix}1_2&0_2\\0_2&0_2\end{bmatrix}\alpha^*(z,w),
$$
where $\alpha(z,w)$ is defined as in \ref{5R3}.

Define
$$v(z,w,e^{i\theta})=\alpha^*(z,w)u_1(z,w,e^{i\theta})\alpha(z,w).$$
Then we have the following property
$$\hspace{-1.7in}{\rm (i)}\,\,\,\,\hspace{0.5in} v(z,w,e^{i\theta})=\begin{bmatrix}1_2&0_2\\0_2&0_2\end{bmatrix}\quad\rforal (z,w)\in \partial D^4$$
and therefore $v$ can be regarded as a map from $S^4\times S^1$ to $\M_4(\C)$.
Moreover,
$${\rm(ii)}\,\,\,\,\hspace{0.5in} v(z,w,e^{i\theta})=\begin{bmatrix}1_2&0_2\\0_2&0_2\end{bmatrix}v(z,w,e^{i\theta})
\begin{bmatrix}1_2&0_2\\0_2&0_2\end{bmatrix}
\rforal (z,w,e^{i\theta})\in S^4\times S^1.$$

By considering the upper left corner of $v$ (still denoted by $v$), we obtain a unitary $v: S^4\times S^1\to SU(2)$. By  \ref{5L55} and  \ref{R56}, $v^4$ is homotopically trivial. Furthermore, by Fact 3 of \ref{5R1}, there is a homotopy
$F: S^4\times S^1\times [0,1]\to SU(2)$ such that
 \beq\label{5T58-4}
 \hspace{-1in}&&{\rm (iii)}\,\,\,\, F(z,w,e^{i\theta},0)=v^4(z,w,e^{i\theta})
\rforal (z,w)\in S^4\andeqn e^{i\theta}\in S^1,\\
\hspace{-2.6in}&&{\rm (iv)}\,\,\, F(\xi,e^{i\theta},t)=1_2\quad\text{ for all }e^{i\theta}\in S^1\andeqn\\
\hspace{-1.9in}&&{\rm (v)}\,\,\,\, F(z,w,e^{i\theta},1)=1_2\,\rforal (z,w)\in S^4, e^{i\theta}\in S^1.
\eneq
Define $G:D^4\times S^1\times [0,1]\to \M_4(\C)$ by
$$G(z,w,e^{i\theta},t)=\alpha(z,w)\begin{bmatrix}F(z,w,e^{i\theta},t)&0_2\\0_2&0_2\end{bmatrix}\alpha^*(z,w).$$
Then by $(iv)$, for $(z,w)\in\partial D^4$, we have
$$G(z,w,e^{i\theta},t)=\begin{bmatrix}1_2&0_2\\0_2&0_2\end{bmatrix}.$$
Hence $G$ defines a map (still denoted by $G$) from $S^4\times S^1\times [0,1]\to \M_4(\C)$. Furthermore $G(z,w,e^{i\theta},t)\in P((z,w)M_4(\C)P(z,w))$, and
$$
G((z,w),e^{i\theta},0)=\alpha(z,w)\begin{bmatrix}v^4&0_2\\0_2&0_2\end{bmatrix}\alpha^*(z,w)\\
                    =u_1^4.
$$
That is $G$ defines a homotopy between $u_1^4$ and the unit $P\in P\M_4C(S^4\times S^1)P$. Consequently
$[u_1^4]=0$ and $[u_1]=0\in K_1(P\M_4C(S^4\times S^1)P)$. Also $[u]=0\in K_1(C(S^4\times S^1))$ as desired.
\end{proof}

\begin{NN}\label{5R59}
{\rm
We identify $P\M_4(C(S^4\times S^1))P$ as a corner of $\M_4C(S^4\times S^1)$, then $K_1(P\M_4C(S^4\times S^1)P)$ is
isomorphic to $K_1(C(S^4\times S^1))=\Z\oplus\Z$ naturally. Let $a\in P\M_4C(S^4\times S^1)P$ be defined by
$$a(x,e^{i\theta})=e^{i\theta}P(x).$$
On the other hand, $a$ could also be regarded as a unitary in $\M_4(C(S^4\times S^1))$ as
$a(x,e^{i\theta})=e^{i\theta}P(x)+(1_4-P(x))$. Then $[a]=(2,1)\in\Z\oplus\Z\cong K_1(C(S^4\times S^1))$,
since $[a]$ is the image of $[P]\in K_0(C(S^4))$ under the exponential map
$$
K_1(C(S^4))\to K_1(C_0(S^1,C(S^4)))
$$
and $[P]=(2,1)\in K_0(C(S^4))\cong\Z\oplus\Z$.
}
\end{NN}

\begin{thm}\label{5T510}
 No element $(1,k)\in K_1(C(S^4\times S^1))$ can be realized by a unitary $b\in P\M_4(C(S^4\times S^1))P$.
\end{thm}
\begin{proof}
We argue for a contradiction. Assume $b\in P\M_4(C(S^4\times S^1))P$ satisfies $[b]=(1,k)\in K_1(P\M_4(C(S^4\times S^1)P))$. Without loss of generality, we assume $b(\xi,1)=P$. Then
$$[b^2a^*]=(0,2k-1)\in K_1(PM_4(C(S^4\times S^1))P).$$
In particular, the map
$$e^{i\theta}\to\text{det}\begin{bmatrix}P(\xi)(b^2a^*)
(\xi,e^{i\theta})P(\xi)&0\\0&1-P(\xi)\end{bmatrix}_{4\times 4}$$
has winding number zero. That is, it is homotopically trivial. Hence
$$
(x,e^{i\theta})\stackrel{h}\longrightarrow \text{det}\begin{bmatrix}P(\xi)(b^2a^*)
(x,e^{i\theta})P(\xi)&0\\0&1-P(\xi)\end{bmatrix}_{4\times 4}
$$
defines a map $h: S^4\times S^1\to S^1$ satisfying $h_*: \pi_1(S^4\times S^1)\to \pi_1(S^1)$ being a zero map. Hence
there is a lifting $\tilde{h}:S^4\times S^1\to \R$ with $h(x,e^{i\theta})=e^{i\tilde{h}(x,e^{i\theta})}$. Define a
unitary $b_1\in P\M_4(C(S^4\times S^1))P$ by $b_1(x,e^{i\theta})=e^{i\frac{1}{2}\tilde{h}(x,e^{i\theta})}P(x)$. Then
$[b_1]=0\in K_1(C(S^4\times S^1))$, and $b^2a^*b_1^*\in U(PM_4C(S^4\times S^1)P)$ has determinant $1$ everywhere. By
Theorem \ref{5T58}, $[b^2a^*b_1^*]=0\in K_1(C(S^4\times S^1))$. On the other hand
$$[b^2a^*b_1^*]=[b^2a^*]=(0,2k-1)\neq 0\in K_1(C(S^4\times S^1)),$$
which is a contradiction.
\end{proof}

\begin{Remark} A similar proof also implies that for any unitary $u\in P\M_4(C(S^4\times S^1))P$, $[u]=l[a]=(2l,l)\in
K_1(C(S^4\times S^1))$ for some $l\in\Z$.
\end{Remark}


\begin{Cor}\label{5CC}
Let $A=C_0(S^1,PC(S^4)P)$ and ${\tilde A}$ be the unitization of $A.$  Then there
is no unitary $u\in A$ such that $[u]=(1,k)\in K_1(A).$ In particular,
no unitary $u$ can be corresponds to a rank one projection in $\M_4(C(S^4)).$
\end{Cor}

\begin{proof}
Note that, as \ref{R57}, we may view $P$ as a projection in
$\M_4(C(S^4\times S^1))$ which is constant along the direction of $S^1.$ So we may view ${\tilde A}$ is a unital
$C^*$--subalgebra of $P\M_4(C(S^4\times S^1))P.$ Thus, by the identification (\ref{R57-1}) in \ref{R57},
Theorem \ref{5T510} applies.
\end{proof}

\begin{thm}\label{Texam}
Let $A=P\M_4(C(S^4))P.$ Then ${\rm Dur}(A)=2.$
\end{thm}

\begin{proof}
There is a projection $e\in \M_2(A)$ which is unitary equivalent  to a rank one projection in
$\M_8(C(S^4))$
correspond to $(1,0)\in K_0(C(S^4)).$
Let $C=C_0((0,1), A).$
By \ref{5CC}, there is no unitary in ${\tilde C}$ which represents a rank one projection.
It follows from \ref{Tcountex} that ${\rm Dur}(A)>1.$

However, since $\M_2(C)$ contains a rank one projection (with trace $\dfrac{1}{\,2\,}$) and
$\rho_C(K_0(\M_2(C)))=\dfrac{1}{\,2\,}\Z,$ by part (3) of Theorem \ref{T3},
${\rm Dur}(\M_2(C))=1.$ It follows that ${\rm Dur}(C)=2.$

\end{proof}

\end{document}